\documentclass[11pt]{article}
\usepackage{amsmath,amsthm,amssymb,amsfonts,geometry}
\usepackage{xr-hyper}
\usepackage{hyperref}
\usepackage[ruled,vlined]{algorithm2e}
\usepackage{graphicx}
\graphicspath{{figures/}}

\newtheorem{definition}{Definition}[section]
\newtheorem{remark}[definition]{Remark}
\newtheorem{lemma}[definition]{Lemma}
\newtheorem{assumption}[definition]{Assumption}
\newtheorem{theorem}[definition]{Theorem}
\newtheorem{proposition}[definition]{Proposition}
\newtheorem{corollary}[definition]{Corollary}
\hypersetup{colorlinks=true,linkcolor=blue,citecolor=blue,urlcolor=blue}

\newcommand{\dist}{\textrm{dist}}

\title{\textbf{Certified spectral approximation of transfer operators and the Gauss map}}
\author{Isaia Nisoli}
\date{}

\begin{document}
\maketitle

\begin{abstract}
Existing rigorous frameworks for transfer operators require spectral gaps
as a~priori input.
We reverse the logical direction: spectral gaps, eigenvalue multiplicities,
and projector stability are certified a~posteriori from finite-rank
computed data, with no spectral assumption beyond a computable
approximation bound.

The method applies to compact operators on a single Banach space and to
quasi-compact operators satisfying a Doeblin--Fortet--Lasota--Yorke
inequality, extending Li's resolution of the Ulam conjecture from the
invariant density to the entire discrete spectrum with certified error
bounds at every finite truncation.
For the discrete spectrum of transfer operators, this resolves a
computability obstruction identified in the Solvability Complexity Index
hierarchy: the dynamical structure furnishing the truncation bound is
precisely the oracle that breaks the barrier preventing spectral
certification from finite-dimensional sections.

As a benchmark, we certify the first $50$ nonzero eigenvalues of the
Gauss--Kuzmin--Wirsing operator to at least $90$ rigorous decimal digits,
together with eigenvectors, Riesz projectors, and spectral gap, providing
a rigorous answer to the Gauss--Babenko--Knuth problem.
\end{abstract}

\noindent\textbf{MSC 2020:} 37C30, 47A75 (primary); 37M25, 65G20, 11K50, 47A58 (secondary).

\section{Introduction and overview}
\label{sec:intro}

\paragraph{Spectral questions as certified linear algebra.}
Transfer operators govern invariant measures, rates of mixing, and limit theorems
through their spectral data (eigenvalues, called Ruelle--Pollicott resonances in this context, and eigenprojectors),
yet extracting \emph{quantitative} spectral information from an
infinite-dimensional operator remains delicate: one must control discretization/truncation effects,
resolvent growth, and spectral-projector stability with fully explicit constants.

The central message of this paper is that, under broadly available regularization hypotheses
(compactness/nuclearity on a suitable space, or a Lasota--Yorke inequality),
\emph{the spectral analysis of an infinite-dimensional operator can be reduced to a finite-dimensional,
computationally certifiable task}.
Starting from a finite-rank discretization $L_N$ of an operator $L$ with an explicit error bound,
we compute spectral objects for $L_N$ (eigenvalues, resolvent bounds, Riesz projectors) and \emph{transfer}
them to $L$ via a resolvent perturbation bound with fully explicit constants.
The analytic work is confined to producing a few
quantitative constants, after which the spectral certification becomes validated numerical linear algebra.
The a~posteriori approach to certifying spectral data of transfer
operators from computed information originates
in~\cite{GalatoloNisoli2014,GalatoloEtAl2015,GalatoloNisoli2016,GalatoloEtAl2020,GalatoloEtAl2023},
which developed rigorous certification of invariant measures,
mixing rates, and escape rates via Lasota--Yorke inequalities
and the coarse-fine approximation strategy.
A subsequent work~\cite{BlumenthalNisoliTaylorCrush2025} extended the approach to
eigenvalue enclosures via validated linear algebra.
Here we isolate a set of hypotheses applicable in a range of
functional-analytic settings, extend the framework to the
Doeblin--Fortet--Lasota--Yorke (DFLY) setting,
and develop high-precision implementations that certify \emph{many}
eigenvalues, eigenvectors, and Riesz projectors simultaneously.

Over the past two decades, anisotropic Banach spaces and Lasota--Yorke inequalities have been
constructed for broad classes of hyperbolic and piecewise hyperbolic
systems~\cite{BlankKellerLiverani2002,GouezelLiverani2006,DemersLiverani2008,BaladiTsujii2007},
yet for many concrete models (higher-dimensional piecewise hyperbolic maps,
billiards with complex geometry, perturbations of known
systems~\cite{DemersKiamariLiverani2021,ButterleyKiamariLiverani2022}), proving
a spectral gap on these spaces remains a difficult open problem.
The present framework reduces the spectral gap question for any such system to a
\emph{finite computation}, producing certified eigenvalue
enclosures and mixing rates.
Even where a spectral gap has been established analytically, the method
provides quantitative data (the actual gap size, subdominant eigenvalues,
and projector norms) that purely qualitative arguments do not.
Crimmins and Froyland~\cite{CrimminsFroyland2020} developed
Fourier-analytic discretizations on the Gou\"ezel--Liverani spaces
and proved spectral convergence for Anosov maps on tori,
but their framework assumes the spectral gap a~priori
(the operator is ``simple quasi-compact'' in the Keller--Liverani sense)
and establishes convergence of the approximations without
explicit error bounds at finite truncation.
The present work removes both limitations: the spectral gap itself
is certified a~posteriori from computed data, with explicit error bounds
at any given matrix size.

\paragraph{The Gauss--Kuzmin--Wirsing problem.}
As a benchmark, we consider one of the oldest problems in ergodic theory and thermodynamic formalism:
the Gauss map $T(x)=\{1/x\}$ on $(0,1]$, which generates continued fractions.
Let $x_0$ be uniformly distributed on $[0,1]$, and set $x_n:=T^n(x_0)$.
Gauss~\cite{GaussWerke} (English translation in~\cite[Appendix~III]{Uspensky1937}) asked for the distribution function
\[
F_n(x)=\mathbb P(x_n\le x),
\]
whose limit $F_\infty(x)=\log_2(1+x)$ corresponds to the invariant probability measure of $T$.
This distribution is governed by iterates of the transfer operator
\[
(\mathcal L f)(x)=\sum_{k\ge1}\frac{1}{(x+k)^2}\,f\!\left(\frac{1}{x+k}\right),
\qquad
F_n(x)=\int_0^x (\mathcal L^n \mathbf 1)(t)\,dt .
\]
The Gauss--Kuzmin--L\'evy problem concerns the rate at which $F_n$ converges to $F_\infty$.
Early quantitative bounds go back to Kuz'min~\cite{kuzmin1928} and L\'evy~\cite{levy1929}, while
Wirsing~\cite{Wirsing1974} identified the leading exponential rate via the subdominant spectrum of the
Gauss--Kuzmin--Wirsing (GKW) operator.
See also Mayer~\cite{mayer1976,mayer1990} and the eigenvalue asymptotics studied by
Alkauskas~\cite{alkauskas2012}.

The problem of quantifying this rate already appears as
Exercise~22 \texttt{[HM50]} in the first edition of Knuth~\cite[p.~337]{KnuthTAOCP2_1e}:
\begin{quote}
\emph{(K.~F.~Gauss.)\ What is the asymptotic behavior of $F_n(x)-\log_2(1+x)$
as $n\to\infty$?}
\end{quote}
This predates both the operator-theoretic resolution of Wirsing~\cite{Wirsing1974} and
the discretization-based approaches of Babenko~\cite{Babenko1978,BabenkoJurev1978}.
In the third edition~\cite[p.~376]{KnuthTAOCP2}, the exercise is reformulated in spectral
terms, reclassified to \texttt{[HM46]}, and attributed to Babenko:
\begin{quote}
\emph{(K.~I.~Babenko.)\ Develop efficient means to calculate accurate approximations to the
quantities $\lambda_j$ and $\Psi_j(x)$\textup, for small $j\ge 3$ and for $0\le x\le 1$.}
\end{quote}
The shift from Gauss's distributional question to Babenko's spectral formulation
reflects the mathematical progress between editions.

The GKW operator has been the subject of extensive high-precision computations; however, as
emphasized in~\cite{alkauskas2012}, rigorous certification has historically been available only
for the first few digits of the leading subdominant eigenvalues.

\paragraph{Certified spectral data for the Gauss map.}

\begin{theorem}\label{thm:main-gkw}
The first $50$ nonzero eigenvalues of the Gauss--Kuzmin--Wirsing operator,
realized as $L:=S\mathcal L:H^2(D_1)\to H^2(D_1)$
\textup{(}Section~\textup{\ref{sec:gkw-hardy})}, are real and simple.
Certified enclosures are given in the supplementary material,
each with at least $90$ certified decimal digits.
In particular, the subdominant eigenvalue (the Gauss--Kuzmin--Wirsing constant) satisfies
$|\lambda_2 - \tilde\lambda_2| < 10^{-175}$, where
\begin{align*}
\tilde\lambda_2 = -0.\,&30366\,30028\,98732\,65859\,74481\,21901\,55623\,31108\,77352\,25365\\
&78951\,88245\,48146\,72269\,95294\,24691\,09843\,40811\,93436\,36368\\
&11098\,27226\,37106\,16938\,47461\,48597\,45801\,31606\,52653\,81818\\
&23787\,91324\,46139\,89647\,64297.
\end{align*}
Certified enclosures for the corresponding eigenvectors and Riesz spectral projectors
are provided, yielding a validated spectral expansion
\[
L^n \mathbf{1}
= \sum_{j=1}^{50} \lambda_j^n \, \ell_j(\mathbf{1}) \, v_j + R_{50}(n),
\qquad
\|R_{50}(n)\|_{H^2(D_1)} \le 3.72\times 10^{21}\cdot (1.01\times 10^{-21})^{n+1},
\]
for the Gauss--Kuzmin distributions, with rigorous error below $10^{-20}$ already at $n=1$.
\end{theorem}

Theorem~\ref{thm:main-gkw} provides rigorous certified solutions to the Gauss--Babenko--Knuth problem:
we compute $\lambda_j$ and $\Psi_j$ for $j=1,\ldots,50$ with at least $90$ validated digits each,
together with the spectral projectors needed for the expansion of $F_n$.
Both the number of eigenvalues and the precision can be pushed further:
the truncation error decays geometrically in the matrix size,
so additional digits or eigenvalues require only larger matrices
and higher-precision arithmetic, with no new analytic input.

\begin{figure}[htbp!]
\centering
\includegraphics[width=0.9\textwidth]{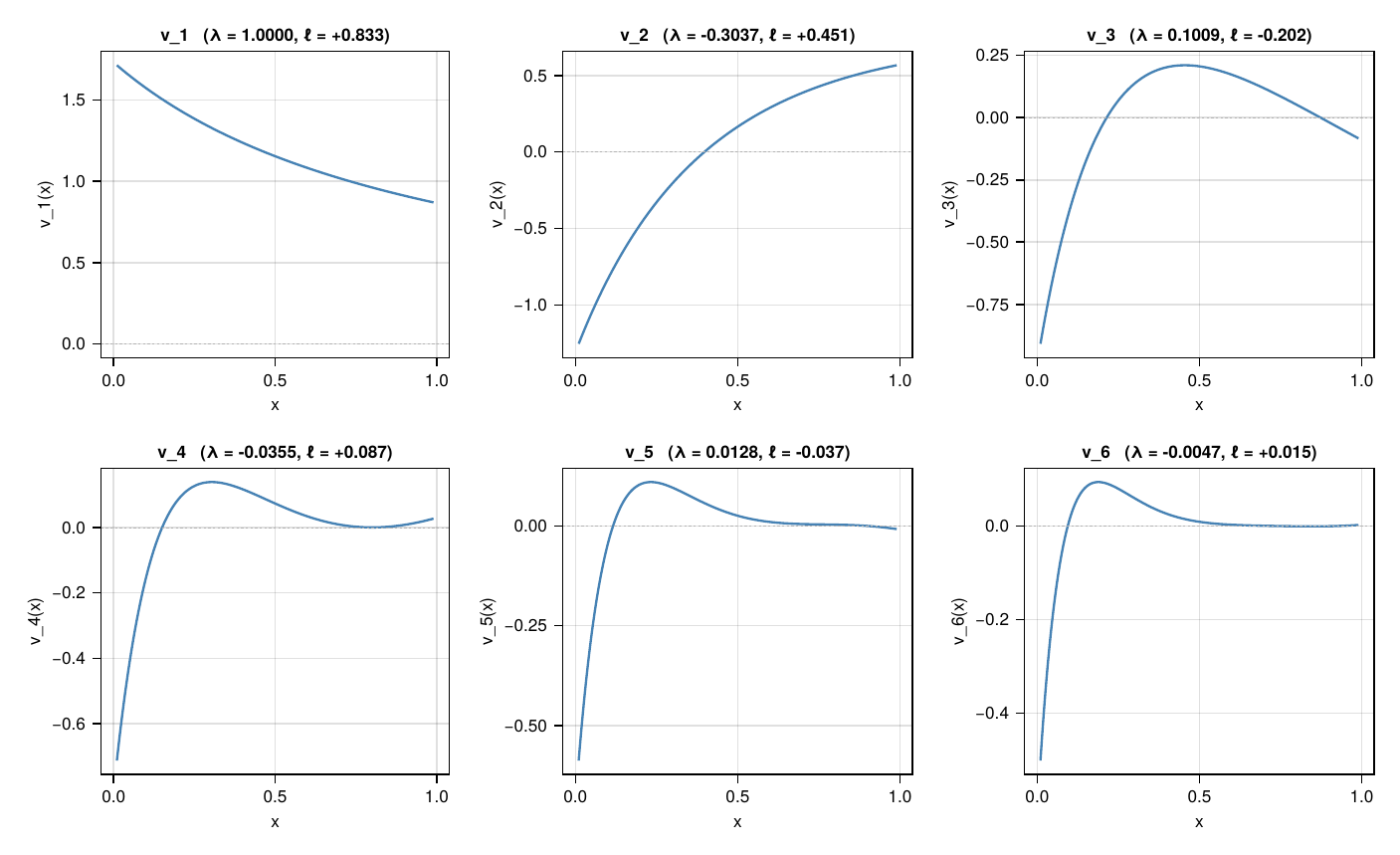}
\caption{The first six certified eigenfunctions $v_1, \ldots, v_6$ of the
GKW transfer operator on $[0,1]$, computed at $K = 512$ (1024-bit
precision).  Each panel shows the eigenvalue $\lambda_j$ and spectral
coefficient $\ell_j(\mathbf{1})$.  The invariant density
$v_1 = \frac{1}{\ln 2}\cdot\frac{1}{1+x}$ (top left) is positive;
higher eigenfunctions develop a singularity at $x = 0$ and
oscillations that reflect the Markov partition of the Gauss map.}
\label{fig:eigenfunctions}
\end{figure}

\paragraph{Spectral gap as an \emph{output}, not an input.}
Spectral separation is certified \emph{a posteriori}
from computed data, rather than assumed \emph{a priori}.
Many rigorous computational frameworks in thermodynamic formalism require a qualitative spectral gap
as a starting hypothesis, from which quantitative consequences are then
derived~\cite{PollicottSlipantschuk2024,PollicottVytnova2022,VytnovaWormell2025}.
The present method certifies spectral gaps directly:
the resolvent bounds along contours that enclose individual eigenvalues simultaneously establish
spectral separation, eigenvalue multiplicity, and projector stability.
This provides the certified spectral data that such approaches take as input.
For Markov operators, this has a direct probabilistic interpretation:
the spectral gap governs the mixing rate of the associated chain,
so a~posteriori spectral certification turns mixing rates into
certified outputs rather than assumed inputs.

\paragraph{Computability of isolated spectral data.}
A longstanding open problem, raised by
Arveson~\cite{Arveson1994} and B\"ottcher~\cite{Bottcher1996},
asks whether the spectrum of a bounded operator can be computed
from its finite-dimensional sections.
The Solvability Complexity Index (SCI) hierarchy developed by
Colbrook, Hansen, and
collaborators~\cite{BenArtziEtAl2020,ColbrookHansen2023,ColbrookRomanHansen2019}
resolves this question by showing that, for general bounded
operators on a separable Hilbert space, computing the spectrum
requires several successive limits and no single algorithm can
provide error control at any finite step.
In particular, computing the discrete spectrum and determining
spectral gaps are both $\Sigma_2^A$-complete: they require two
successive limits, and no algorithm taking only matrix entries as
input can certify that its output is within a prescribed distance
of the true spectral data~\cite[Theorems~3.13 and~3.15]{ColbrookHansen2023}.

The a~posteriori framework of
Theorems~\ref{thm:spectral-convergence-compact}
and~\ref{thm:spectral-convergence-dfly} resolves the problem in
the affirmative for the discrete spectrum of transfer operators
with explicit truncation bounds.
For each isolated eigenvalue, the certification algorithm takes
a single limit ($K\to\infty$) and provides verified error
bounds at every truncation where the $z$-dependent perturbation
condition
$\varepsilon_K\,\mathcal R(z,L_K)<1$
is satisfied along the enclosing contour.
The error bound at each eigenvalue depends on~$z$ through the
resolvent of the finite-rank model~$L_K$, so certification
targets individual eigenvalues via contour-specific resolvent
control rather than uniform approximation of the full spectrum.
The input that circumvents the general impossibility is the
computable truncation bound~$\varepsilon_K$,
which is not available for abstract compact operators given
only by matrix entries but is furnished by the
dynamical structure of the transfer-operator setting.
Since every compact operator with positive spectral radius is
quasi-compact (its essential spectral radius is zero),
the impossibility results for general compact operators
suggest that the Lasota--Yorke inequality is a strictly stronger
condition than quasi-compactness in a computational sense:
it provides precisely the oracle enrichment needed to break
the SCI barrier for spectral certification,
and the impossibility results indicate that this structure
cannot be universally available or uniformly recoverable
from finite-dimensional data.

\paragraph{The general framework.}
Although the GKW operator is the main benchmark, the method applies
whenever the following inputs are available for a bounded
operator $L:\mathcal B\to\mathcal B$ on a Banach space:
\textup{(A)}~a finite-rank approximation $L_N$ with an explicit error bound
$\|L-L_N\|\le\varepsilon_N$;
\textup{(B)}~certified resolvent bounds $\|(zI-L_N)^{-1}\|$ along contours
$\Gamma\subset\mathbb C$, obtained by validated linear algebra (Schur/SVD, interval/ball arithmetic);
and
\textup{(C)}~a resolvent perturbation bound transferring certified bounds from $L_N$ to $L$.
The key condition is
$\varepsilon_N\sup_{z\in\Gamma}\|(zI-L_N)^{-1}\|<1$;
when it holds, spectral enclosures, Riesz projectors, and invariant subspaces of $L_N$ transfer to $L$
with explicit error control.
The full framework is developed in Section~\ref{sec:compact-cert};
holomorphic transfer operators provide a natural setting where \textup{(A)} arises from
compactness/nuclearity induced by holomorphic extension, and related analytic frameworks appear in,
e.g.,~\cite{BandtlowJenkinson2008,PollicottSlipantschuk2024}.
For the Gauss map, the holomorphic contraction of the branches induces
compactness on Hardy spaces with explicit truncation bounds
(input~\textup{(A)} with exponential decay $(2/3)^K$), and certified Schur decompositions
provide the resolvent control (input~\textup{(B)}).
In an appendix we show that the same philosophy extends to the
strong-weak setting of Doeblin--Fortet--Lasota--Yorke (DFLY)
inequalities, where the operator $L$ acts on a Banach space
$\mathcal B_s$ equipped with a weaker norm $\|\cdot\|_w$ and
satisfies a Lasota--Yorke inequality
$\|L f\|_s \le a \|f\|_s + b\|f\|_w$.
Input~\textup{(A)} becomes the strong-to-weak approximation bound
$\|L - L_k\|_{s\to w} \le \delta_k$, and the resolvent of $L_k$ is
first controlled in the weak norm, which is computable from a finite
matrix, and then lifted to the strong norm via the Lasota--Yorke
inequality.

\paragraph{Guaranteed convergence of the full spectral picture.}
A second main contribution, beyond the GKW application, is the following
convergence guarantee.

\begin{theorem}\label{thm:spectral-approx-intro}
Let $L$ be a bounded operator on a Banach space, and let $(L_K)$ be a sequence of
finite-rank discretizations with $\|L-L_K\|\le\varepsilon_K\to 0$
\textup{(compact case, Theorem~\ref{thm:spectral-convergence-compact})}
or $\|L-L_k\|_{s\to w}\le\delta_k\to 0$ under a DFLY inequality,
subject to a rate condition linking $\delta_k$ to the norm-equivalence
constants of the discretization spaces
\textup{(strong-weak case, Theorem~\ref{thm:spectral-convergence-dfly})}.
Let $U\subset\mathbb C$ be a spectral window isolating finitely many eigenvalues of $L$
outside the essential spectral radius.
Then for $K$ large enough:
\begin{enumerate}
\item \textup{(No spectral pollution.)}
$L_K$ and $L$ have the same number of eigenvalues
\textup{(counted with multiplicity)} in $U$.
\item \textup{(Projector convergence.)}
$\|P_{L_K}(U)-P_L(U)\|\to 0$ as $K\to\infty$.
\item \textup{(A~posteriori certification.)}
The perturbation condition that triggers \textup{(1)--(2)} is
checkable from the computed data of $L_K$ alone, with no spectral gap
assumption.
\end{enumerate}
In the compact case, every isolated nonzero eigenvalue, eigenvector, and Riesz projector
of $L$ can be approximated to arbitrary precision by the finite-rank data of $L_K$.
In the DFLY case, the same holds for eigenvalues whose modulus exceeds
any fixed $\mu>a$, provided the rate condition is satisfied at $\mu$.
\end{theorem}

These results are stated in abstract operator-theoretic terms.
For quasi-compact operators satisfying a DFLY inequality,
the spectrum outside the essential spectral radius consists of
isolated eigenvalues of finite multiplicity, and the discrete
spectral data there can be recovered from finite-rank
discretizations, subject to a rate condition on the
discretization error.
Theorem~\ref{thm:spectral-convergence-dfly} makes this precise
and quantitative.
Li~\cite{Li1976} proved convergence of Ulam approximations to the
invariant density for piecewise expanding maps;
the Keller--Liverani perturbation theorem~\cite{KellerLiverani1999}
extends this to the general DFLY framework
(see, e.g.,~\cite{Liverani2001,GalatoloEtAl2023})
and implies that isolated eigenvalues outside the essential spectral
radius are stable under discretization, qualitatively extending
convergence from the invariant density to the full discrete spectrum.
The present work makes this extension \emph{certified and quantitative}:
the Galerkin hierarchy $L_k$ approximates all eigenvalues,
eigenvectors, and spectral projectors outside the essential spectral
radius, with explicit error bounds at every finite truncation.

The Keller--Liverani spectral stability theorem~\cite{KellerLiverani1999}
and related frameworks~\cite{DemersKiamariLiverani2021,ButterleyKiamariLiverani2022}
assume a~priori that certain eigenvalues are isolated and deduce their
stability under perturbation.
Liverani~\cite{Liverani2001} reversed the logical direction by proposing
to use the finite-rank (Ulam) approximation as the ``known'' operator
and certify spectral data of the original operator a~posteriori,
via the Keller--Liverani perturbation theory;
this was presented as a feasibility study without rigorous computations,
and the required strong resolvent norms of the discretization
are nontrivial to compute in practice.
The present work addresses these difficulties with the
a~posteriori certification tools described above.
In both settings the computable truncation bound is of dynamical origin:
in the compact case it arises from the regularity of the underlying map
(nuclearity, domain gain on Hardy spaces),
while in the DFLY case it requires hyperbolicity constants,
adapted Banach spaces, and a Lasota--Yorke inequality.
Once these bounds are in hand, the entire spectral
certification is reduced to a computer program whose output is a rigorous proof.

\paragraph{Related approaches.}
A substantial literature develops rigorous computer-assisted estimates for transfer operators.
Representative examples include explicit eigenvalue bounds on holomorphic function
spaces~\cite{BandtlowJenkinson2007,BandtlowJenkinson2008},
rigorous schemes based on Lagrange--Chebyshev approximation and min-max
techniques~\cite{SlipantschukEtAl2013,PollicottSlipantschuk2024},
and thermodynamic computations for Hausdorff dimension
problems~\cite{PollicottVytnova2022,VytnovaWormell2025}.
As noted above, such methods typically assume spectral separation as input;
the a~posteriori approach developed here and
in~\cite{GalatoloNisoli2014,GalatoloEtAl2015,GalatoloNisoli2016,GalatoloEtAl2020,GalatoloEtAl2023,BlumenthalNisoliTaylorCrush2025}
provides precisely that input, suggesting a \emph{hybrid} strategy: a~posteriori certification via
coarse (e.g.\ Ulam) discretizations can establish spectral gaps and
isolation, after which exponentially convergent methods (e.g.\ on
holomorphic function spaces) can be deployed for high-precision
computation.
In the single-space (compact) setting, this can even bootstrap: a crude
initial bound on the spectral gap can be sharpened a~posteriori to
arbitrary precision, as we demonstrate for the Gauss map.
In the two-space (DFLY) setting, however, a~priori methods that
assume a spectral gap often cannot certify it from their own output,
making the a~posteriori approach essential in many cases.

A parallel program addresses the rigorous computation of spectra and
pseudospectra for general classes of operators, including verified
algorithms for spectral inclusion
sets~\cite{ColbrookHansen2023,ChandlerWildeEtAl2024}
and data-driven spectral methods with convergence guarantees
for Koopman and transfer
operators~\cite{ColbrookTownsend2024,BoulleColbrookConradie2025}.
These approaches typically work in Hilbert-space ($L^2$ or $\ell^2$)
settings with operators specified by matrix entries or trajectory data,
whereas the present method exploits explicit dynamical structure
(regularity and domain gain, DFLY inequalities)
for operator-norm certification of isolated spectral data with
arbitrary precision.

\paragraph{Organization.}
Section~\ref{sec:compact-cert}: abstract certified-spectral template.
Section~\ref{sec:validated-linalg}: certified resolvent bounds via validated Schur decompositions.
Section~\ref{sec:gkw-hardy}: Hardy-space framework and truncation bounds for the Gauss operator.
Section~\ref{sec:gkw-certified}: certified spectral enclosures and eigendata.
Appendix: DFLY resolvent lifting.

\section{A posteriori validated spectral certification: an abstract template}
\label{sec:compact-cert}

The resolvent perturbation bound at the core of this section can be viewed as a quantitative,
a posteriori counterpart of the Keller--Liverani spectral stability
theorem~\cite{KellerLiverani1999}: rather than deducing spectral convergence from an
abstract approximation scheme, we extract explicit enclosures from validated resolvent data.
We work in the simplest setting: a compact operator on a single Banach space $\mathcal B$.
The strong-weak (DFLY) framework is treated in Appendix~\ref{app:dfly}.

\begin{assumption}\label{ass:compact-discretization}
Let $\mathcal B$ be a complex Banach space and let $L\in\mathcal L(\mathcal B)$ be compact.
Assume we are given a sequence of finite-rank operators $\pi_K\in\mathcal L(\mathcal B)$ and define
\[
L_K := \pi_K\,L\,\pi_K\qquad (K\ge 1).
\]
We assume:
\begin{enumerate}
\item Each $\pi_K$ is a projection with $\pi_K^2=\pi_K$, $\|\pi_K\|\le 1$, and the family is nested:
for all $K'\ge K$,
\[
\pi_K\pi_{K'}=\pi_K.
\]
or equivalently, $\operatorname{Ran}\pi_K\subset \operatorname{Ran}\pi_{K'}$.

\item There exists $C\ge 0$ such that
\[
\|L\|_{\mathcal B\to\mathcal B}\le C,
\qquad
\|L_K\|_{\mathcal B\to\mathcal B}\le C
\quad\text{for all }K\ge 1.
\]

\item There exists a computable sequence $\varepsilon_K\downarrow 0$ such that
\[
\|(\mathrm{Id}-\pi_K)L\|_{\mathcal B\to\mathcal B}
\;+\;
\|L(\mathrm{Id}-\pi_K)\|_{\mathcal B\to\mathcal B}
\ \le\ \varepsilon_K
\quad\text{for all }K\ge 1.
\]
In particular,
\[
\|L-L_K\|_{\mathcal B\to\mathcal B}\le \varepsilon_K,
\qquad
\|L_{K'}-L_K\|_{\mathcal B\to\mathcal B}\le \varepsilon_K
\quad\text{for all }K'\ge K.
\]
\end{enumerate}
\end{assumption}

\subsection{Resolvent inequalities and the Householder argument}
\label{subsec:resolvent-householder}

Throughout, $\|\cdot\|$ denotes the operator norm on $\mathcal B$.
For $A\in\mathcal L(\mathcal B)$ we write
$R_A(z):=(z\,\mathrm{Id}-A)^{-1}$ for the resolvent ($z\in\rho(A)$),
$\mathcal R(z,A):=\|R_A(z)\|$ for the resolvent norm
(with $\mathcal R(z,A):=+\infty$ for $z\in\sigma(A)$), and
$\Lambda_{\varepsilon}(A):=\{z\in\mathbb C:\mathcal R(z,A)\ge 1/\varepsilon\}$
for the $\varepsilon$-pseudospectrum.

\begin{lemma}[{\cite[Ch.~IV, \S1.1]{Kato1995}}]\label{lem:first-resolvent-ineq}
Let $A,B\in\mathcal L(\mathcal B)$ and $z\in\rho(A)\cap\rho(B)$. Then
\[
R_A(z)-R_B(z)=R_A(z)\,(A-B)\,R_B(z),
\]
and consequently
\[
\|R_A(z)-R_B(z)\|
\ \le\
\mathcal R(z,A)\,\|A-B\|\,\mathcal R(z,B).
\]
\end{lemma}

\begin{corollary}\label{cor:resolvent-bridge}
Let $A,B\in\mathcal L(\mathcal B)$ and $z\in\rho(B)$.
If
\[
\alpha(z):=\|(A-B)R_B(z)\|<1,
\]
then $z\in\rho(A)$ and
\[
\mathcal R(z,A)\le \frac{\mathcal R(z,B)}{1-\alpha(z)},
\qquad
\|R_A(z)-R_B(z)\|
\le \frac{\mathcal R(z,B)^2\,\|A-B\|}{1-\alpha(z)}.
\]
\end{corollary}

\begin{proof}
Factor
\[
z\,\mathrm{Id}-A=\bigl(\mathrm{Id}-(A-B)R_B(z)\bigr)\,(z\,\mathrm{Id}-B),
\]
invert the first term by Neumann series, and use Lemma~\ref{lem:first-resolvent-ineq}.
\end{proof}

\begin{proposition}\label{prop:pseudo-enclosure-compact}
Let $L\in\mathcal L(\mathcal B)$ be compact and let $L_K$ be finite rank with
\[
\|L-L_K\|\le \varepsilon_K.
\]
Then
\[
\sigma(L)\ \subset\ \Lambda_{\varepsilon_K}(L_K).
\]

Moreover, let $U$ be a bounded connected component of $\Lambda_{\varepsilon_K}(L_K)$ with
$0\notin\overline U$, and assume $\partial U$ is a piecewise $C^1$ Jordan curve. Then
$\partial U\subset\rho(L_K)$ and
\[
\sup_{z\in\partial U}\ \varepsilon_K\,\mathcal R(z,L_K)\ <\ 1,
\]
so $\partial U\subset\rho(L)$, and $L$ and $L_K$ have the same total algebraic multiplicity of
eigenvalues in $U$:
\[
\dim \operatorname{Ran} P_L(U)\ =\ \dim \operatorname{Ran} P_{L_K}(U),
\qquad
P_A(U):=\frac{1}{2\pi i}\int_{\partial U} (zI-A)^{-1}\,dz.
\]
\end{proposition}

\begin{proof}
Let $\lambda\in\sigma(L)$. If $\lambda\notin\Lambda_{\varepsilon_K}(L_K)$, then
$\lambda\in\rho(L_K)$ and $\mathcal R(\lambda,L_K)<1/\varepsilon_K$, hence
\[
\|(L-L_K)(\lambda I-L_K)^{-1}\|
\le \varepsilon_K\,\mathcal R(\lambda,L_K)
<1.
\]
Thus $\lambda I-L=\bigl(I-(L-L_K)(\lambda I-L_K)^{-1}\bigr)(\lambda I-L_K)$ is invertible,
contradicting $\lambda\in\sigma(L)$. This proves $\sigma(L)\subset\Lambda_{\varepsilon_K}(L_K)$.

For the multiplicity statement, the definition of $U$ implies
$\sup_{z\in\partial U}\mathcal R(z,L_K)<1/\varepsilon_K$, hence $\partial U\subset\rho(L)$ by
Corollary~\ref{cor:resolvent-bridge}. Consider $L_t:=L_K+t(L-L_K)$, $t\in[0,1]$. For $z\in\partial U$,
\[
\|(L_t-L_K)(zI-L_K)^{-1}\|
\le t\,\varepsilon_K\,\mathcal R(z,L_K)
<1,
\]
so $\partial U\subset\rho(L_t)$ for all $t$. Therefore $P_{L_t}(U)$ varies continuously with $t$ in
operator norm, and $\dim\operatorname{Ran}P_{L_t}(U)$ is constant. Evaluating at $t=0$ and $t=1$
gives the claim.
\end{proof}

\subsection{Resolvent perturbation bound and spectral exclusions}
\label{subsec:resolvent-exclusions}

\begin{definition}\label{def:spectral-exclusion-curve}
Let $A\in\mathcal L(\mathcal B)$ and let $\Gamma$ be a piecewise $C^1$ closed curve in $\mathbb C$.
We say that $\Gamma$ is a \emph{resolvent curve} for $A$ if $\Gamma\subset\rho(A)$.
If $\Gamma$ bounds a domain $U$ and $\Gamma\subset\rho(A)$, we also call $U$ a \emph{spectral window} for $A$.
\end{definition}

\begin{proposition}\label{prop:exclude-by-bridge}
Assume $\|L-L_K\|\le \varepsilon_K$ and let $\Gamma\subset\rho(L_K)$ be a closed curve such that
\[
\sup_{z\in\Gamma}\ \varepsilon_K\,\mathcal R(z,L_K)\ <\ 1.
\]
Then $\Gamma\subset\rho(L)$ as well. In particular, if $\Gamma=\partial U$ for a bounded domain $U$,
then $\sigma(L)$ cannot intersect $\Gamma$ and the Riesz projector $P_L(U)$ is well defined.
\end{proposition}

\begin{proof}
For each $z\in\Gamma$ we have
\[
\|(L-L_K)R_{L_K}(z)\|\le \varepsilon_K\,\mathcal R(z,L_K)<1,
\]
so $z\in\rho(L)$ by Corollary~\ref{cor:resolvent-bridge}.
\end{proof}

\subsection{Riesz projectors, multiplicity, and invariant subspaces}
\label{subsec:riesz-projectors}

\begin{definition}\label{def:riesz-projector}
Let $A\in\mathcal L(\mathcal B)$ and let $U\subset\mathbb C$ be a bounded domain such that
$\Gamma:=\partial U\subset\rho(A)$ is a piecewise $C^1$ Jordan curve. The \emph{Riesz projector} onto the
spectral subset $\sigma(A)\cap U$ is
\[
P_A(U):=\frac{1}{2\pi i}\int_{\Gamma} (zI-A)^{-1}\,dz.
\]
Its range $\mathcal E_A(U):=\operatorname{Ran}P_A(U)$ is $A$-invariant. The \emph{total algebraic multiplicity}
of eigenvalues of $A$ in $U$ is $\dim \mathcal E_A(U)$ (finite in the compact setting when $\dist(U,0)>0$).
\end{definition}


\begin{assumption}\label{ass:riesz-standing}
Let $L,L_K\in\mathcal L(\mathcal B)$ satisfy $\|L-L_K\|\le \varepsilon_K$.
Let $U\subset\mathbb C$ be a bounded domain with piecewise $C^1$ Jordan boundary $\Gamma:=\partial U$ such that
$\Gamma\subset\rho(L_K)$ and
\[
\eta
:=\sup_{z\in\Gamma}\ \|(L-L_K)R_{L_K}(z)\|
\le \varepsilon_K\,\sup_{z\in\Gamma}\mathcal R(z,L_K)
<1.
\]
\end{assumption}


\begin{lemma}\label{lem:projector-perturb}
Assume Assumption~\ref{ass:riesz-standing}. Then $\Gamma\subset\rho(L)$ and, writing
\[
P:=P_L(U),\qquad P_K:=P_{L_K}(U),
\qquad
M:=\sup_{z\in\Gamma}\mathcal R(z,L_K),
\]
we have
\[
\|P-P_K\|
\le
\frac{\ell(\Gamma)}{2\pi}\cdot
\frac{\varepsilon_K\,M^2}{1-\eta},
\]
where $\ell(\Gamma)$ denotes the length of $\Gamma$.
\end{lemma}

\begin{proof}
Fix $z\in\Gamma$. Since $z\in\rho(L_K)$ and $\|(L-L_K)R_{L_K}(z)\|\le\eta<1$, the Neumann factorization
\[
zI-L=\bigl(I-(L-L_K)R_{L_K}(z)\bigr)\,(zI-L_K)
\]
shows $z\in\rho(L)$ and
\[
R_L(z)=R_{L_K}(z)\bigl(I-(L-L_K)R_{L_K}(z)\bigr)^{-1}.
\]
Hence
\[
\|R_L(z)\|
\le
\frac{\|R_{L_K}(z)\|}{1-\|(L-L_K)R_{L_K}(z)\|}
\le
\frac{\mathcal R(z,L_K)}{1-\eta}.
\]
Using the resolvent identity
\[
R_L(z)-R_{L_K}(z)=R_L(z)\,(L-L_K)\,R_{L_K}(z),
\]
we obtain
\[
\|R_L(z)-R_{L_K}(z)\|
\le
\|R_L(z)\|\,\|L-L_K\|\,\|R_{L_K}(z)\|
\le
\frac{\varepsilon_K\,\mathcal R(z,L_K)^2}{1-\eta}.
\]
Finally,
\[
\|P-P_K\|
=
\left\|\frac{1}{2\pi i}\int_\Gamma \bigl(R_L(z)-R_{L_K}(z)\bigr)\,dz\right\|
\le
\frac{1}{2\pi}\int_\Gamma \|R_L(z)-R_{L_K}(z)\|\,|dz|
\]
and bounding $\mathcal R(z,L_K)\le M$ along $\Gamma$ yields the claim.
\end{proof}

\begin{remark}\label{rem:rank-stability}
If $\|P-P_K\|<1$, then $P$ restricts to an isomorphism
$\operatorname{Ran}P_K\to\operatorname{Ran}P$; in particular,
$\dim\operatorname{Ran}P=\dim\operatorname{Ran}P_K$.
\end{remark}


\begin{lemma}\label{lem:proj-controls-evec}
Assume Assumption~\ref{ass:riesz-standing} and suppose
\[
\operatorname{rank}P=\operatorname{rank}P_K=1,
\qquad
\vartheta:=\|P-P_K\|<1.
\]
Let $v_K\in\operatorname{Ran}P_K$ with $\|v_K\|=1$.
Then
\[
\|(I-P)v_K\|\le \vartheta,
\qquad
\|Pv_K\|\ge 1-\vartheta.
\]
In particular, the normalized vector
\[
v:=\frac{Pv_K}{\|Pv_K\|}\in\operatorname{Ran}P
\]
is well-defined and satisfies
\[
\|v-v_K\|\le \frac{2\vartheta}{1-\vartheta}.
\]
\end{lemma}

\begin{proof}
Since $v_K\in\operatorname{Ran}P_K$ we have $P_Kv_K=v_K$, hence
\[
(I-P)v_K=(P_K-P)v_K,
\qquad\Rightarrow\qquad
\|(I-P)v_K\|\le \|P-P_K\|\,\|v_K\|=\vartheta.
\]
Thus
\[
\|Pv_K\|\ge \|v_K\|-\|(I-P)v_K\|\ge 1-\vartheta,
\]
so $v$ is well-defined. Next,
\[
\|v-v_K\|
=
\left\|\frac{Pv_K}{\|Pv_K\|}-v_K\right\|
\le
\left\|\frac{Pv_K}{\|Pv_K\|}-Pv_K\right\|
+\|Pv_K-v_K\|.
\]
The second term is $\|Pv_K-v_K\|=\|(P-P_K)v_K\|\le\vartheta$.
For the first term, $\bigl\|\frac{Pv_K}{\|Pv_K\|}-Pv_K\bigr\|=\bigl|1-\|Pv_K\|\bigr|\,\|Pv_K\|^{-1}\,\|Pv_K\|
=\bigl|1-\|Pv_K\|\bigr|\le \|v_K-Pv_K\|=\|(I-P)v_K\|\le\vartheta$,
and dividing by $\|Pv_K\|\ge 1-\vartheta$ yields
\[
\|v-v_K\|\le \frac{\vartheta+\vartheta}{1-\vartheta}=\frac{2\vartheta}{1-\vartheta}.
\]
\end{proof}


\begin{lemma}\label{lem:proj-controls-eval}
Assume the setting of Lemma~\ref{lem:proj-controls-evec} and moreover $\|L_K\|\le C$.
Let $\lambda$ be the (unique) eigenvalue of $L$ in $U$ and $\lambda_K$ the (unique) eigenvalue of $L_K$ in $U$.
Then
\[
|\lambda-\lambda_K|
\ \le\
\frac{\varepsilon_K(1+\vartheta)+2C\,\vartheta}{1-\vartheta}.
\]
\end{lemma}

\begin{proof}
Since $\operatorname{rank}P=\operatorname{rank}P_K=1$, both $\operatorname{Ran}P$ and $\operatorname{Ran}P_K$
are one-dimensional invariant subspaces, and $L|_{\operatorname{Ran}P}$ (resp.\ $L_K|_{\operatorname{Ran}P_K}$)
acts as multiplication by $\lambda$ (resp.\ $\lambda_K$).

Define the linear map
\[
\Phi:=P|_{\operatorname{Ran}P_K}:\operatorname{Ran}P_K\to\operatorname{Ran}P.
\]
For $x\in\operatorname{Ran}P_K$, we have $x=P_Kx$, hence
\[
\|\Phi x-x\|=\|(P-P_K)x\|\le \vartheta\|x\|,
\qquad\Rightarrow\qquad
\|\Phi x\|\ge (1-\vartheta)\|x\|.
\]
Therefore $\Phi$ is invertible and $\|\Phi^{-1}\|\le (1-\vartheta)^{-1}$.
Also, on $\operatorname{Ran}P_K$,
\[
\|\Phi x\|=\|Px\|\le \|x\|+\|(P-P_K)x\|\le (1+\vartheta)\|x\|,
\qquad\Rightarrow\qquad
\|\Phi\|\le 1+\vartheta.
\]

Now observe that for $x\in\operatorname{Ran}P_K$,
\[
\Phi^{-1}L\Phi\,x
\in \operatorname{Ran}P_K,
\qquad\text{and}\qquad
\Phi^{-1}L\Phi\,x=\lambda\,x,
\]
because $\Phi x\in\operatorname{Ran}P$ and $L$ acts as $\lambda$ on $\operatorname{Ran}P$, while $\Phi^{-1}$
maps $\operatorname{Ran}P$ back to $\operatorname{Ran}P_K$.
Hence, as operators on the one-dimensional space $\operatorname{Ran}P_K$,
\[
|\lambda-\lambda_K|
=
\bigl\|\Phi^{-1}L\Phi - L_K|_{\operatorname{Ran}P_K}\bigr\|.
\]
Insert $L=L_K+(L-L_K)$ and add and subtract $\Phi^{-1}L_K\Phi$:
\[
\|\Phi^{-1}L\Phi - L_K\|
\le
\|\Phi^{-1}(L-L_K)\Phi\|
+\|\Phi^{-1}L_K\Phi-L_K\|.
\]
The first term satisfies
\[
\|\Phi^{-1}(L-L_K)\Phi\|
\le
\|\Phi^{-1}\|\,\|L-L_K\|\,\|\Phi\|
\le
\frac{\varepsilon_K(1+\vartheta)}{1-\vartheta}.
\]
For the second term, note that on $\operatorname{Ran}P_K$ we have $\Phi= P$ and $P_K=\mathrm{Id}$, hence
\[
\Phi^{-1}L_K\Phi-L_K
=
\Phi^{-1}\bigl(L_KP-PL_K\bigr).
\]
Since $\Phi^{-1}$ is bounded by $(1-\vartheta)^{-1}$, it remains to bound the commutator:
\[
\|L_KP-PL_K\|
\le
\|L_K(P-P_K)\|+\|(P-P_K)L_K\|
\le
2\|L_K\|\,\|P-P_K\|
\le
2C\,\vartheta.
\]
Putting the two bounds together gives
\[
|\lambda-\lambda_K|
\le
\frac{\varepsilon_K(1+\vartheta)+2C\,\vartheta}{1-\vartheta}.
\]
\end{proof}

\begin{remark}\label{rem:proj-eigenpair}
A certified bound on $\|P-P_K\|$ yields:
(i) a certified \emph{invariant subspace} enclosure $\operatorname{Ran}P_K\approx\operatorname{Ran}P$ (always),
(ii) in the rank-one (simple eigenvalue) case, an explicit bound on the eigenvector direction
(Lemma~\ref{lem:proj-controls-evec}), and
(iii) combined with $\|L-L_K\|$ (and a crude bound on $\|L_K\|$), an explicit eigenvalue error bound
(Lemma~\ref{lem:proj-controls-eval}).
For higher multiplicity, $P$ certifies the whole invariant subspace; individual eigenvectors require
additional structure (e.g.\ normality, a chosen basis, or a spectral gap within the cluster).
\end{remark}

\subsection{Coarse--fine approach}\label{subsec:coarse-fine}

Under Assumption~\ref{ass:compact-discretization}, for all $K'\ge K$,
\begin{equation}\label{eq:coarse-fine-diff}
\|L_{K'}-L_K\|\ \le\ \varepsilon_K,
\qquad\text{and}\qquad
\|L-L_{K'}\|\ \le\ \varepsilon_{K'}\ \le\ \varepsilon_K.
\end{equation}
Thus a resolvent estimate at a coarse level $K$ propagates to any finer $K'\ge K$; the DFLY analogue
requires additional bookkeeping (Appendix~\ref{app:dfly}).

Fix a region $U\subset\mathbb C$ and assume
that $\partial U\subset\rho(L_K)$ and that
\begin{equation}\label{eq:coarse-fine-alpha}
\sup_{z\in\partial U}\ \varepsilon_K\,\mathcal R(z,L_K)\ <\ 1.
\end{equation}
Then, by applying Corollary~\ref{cor:resolvent-bridge} with $A=L_{K'}$ and $B=L_K$ and using
\eqref{eq:coarse-fine-diff}, we obtain $\partial U\subset\rho(L_{K'})$ and the explicit bound
\begin{equation}\label{eq:coarse-fine-resolvent}
\mathcal R(z,L_{K'})\ \le\ \frac{\mathcal R(z,L_K)}{1-\varepsilon_K\,\mathcal R(z,L_K)}
\qquad\text{for all }z\in\partial U.
\end{equation}
In particular, the same coarse-level contour $\partial U$ that isolates (say) a simple eigenvalue of
$L_K$ also isolates the corresponding eigenvalue of $L_{K'}$ for every $K'\ge K$, with no need to
recompute resolvents from scratch at the fine level.

This yields a coarse-fine workflow:
\begin{enumerate}
\item At a modest $K$, compute $\mathcal R(z,L_K)$ and use
Proposition~\ref{prop:pseudo-enclosure-compact} to produce enclosures for $\sigma(L)$ and certify
spectral separation.

\item Increase to $K'\gg K$. Use \eqref{eq:coarse-fine-resolvent} to control
$\mathcal R(z,L_{K'})$ on the same contour $\partial U$, then apply
Lemma~\ref{lem:projector-perturb} with the smaller error $\varepsilon_{K'}$ to obtain tighter
projector bounds.
\end{enumerate}

\begin{remark}\label{rem:coarse-fine-roles}
The two levels play different roles.
At the coarse level $K$, the resolvent bounds need only be good enough to certify
that the contour $\partial U$ lies in $\rho(L)$, i.e., to establish spectral isolation.
The resulting projector bound $\|P-P_K\|$ may be far too large for useful
eigenvector or eigenvalue enclosures.
At the fine level $K'$, the same contour is reused and the projector bound improves
by a factor of $\varepsilon_{K'}/\varepsilon_K$, without recomputing any resolvent.
In particular, the coarse matrix can be small (and hence cheap to certify), while
the fine matrix is large and delivers precision.
\end{remark}

The following theorem shows that the perturbation condition is not merely a checkable
sufficient condition: it is \emph{eventually} satisfied, so the certification
pipeline is guaranteed to succeed at some finite truncation level.

\begin{theorem}\label{thm:spectral-convergence-compact}
Under Assumption~\ref{ass:compact-discretization}, let $U\subset\mathbb C$ be a bounded domain
with piecewise $C^1$ boundary $\Gamma:=\partial U$ such that $0\notin\overline U$,
$\Gamma\subset\rho(L)$, and $\sigma(L)\cap\Gamma=\emptyset$.
Then there exists $K_0$ such that for all $K\ge K_0$:
\begin{enumerate}
\item $\Gamma\subset\rho(L_K)$ and the perturbation condition
$\sup_{z\in\Gamma}\varepsilon_K\,\mathcal R(z,L_K)<1$ holds;
\item $\dim\operatorname{Ran}P_{L_K}(U)=\dim\operatorname{Ran}P_L(U)$;
\item $\|P_{L_K}(U)-P_L(U)\|\to 0$ as $K\to\infty$.
\end{enumerate}
In particular, every eigenvalue, eigenvector, and Riesz projector of $L$ in $U$ can be
approximated to arbitrary precision by the finite-rank data of $L_K$.
\end{theorem}

\begin{proof}
Since $\Gamma$ is compact and $\Gamma\subset\rho(L)$, we have
$M:=\sup_{z\in\Gamma}\mathcal R(z,L)<\infty$.
Because $\|L-L_K\|\le\varepsilon_K\to 0$, the second resolvent identity gives
$\mathcal R(z,L_K)\to\mathcal R(z,L)$ uniformly on $\Gamma$.
In particular, $\sup_{z\in\Gamma}\varepsilon_K\,\mathcal R(z,L_K)\to 0$,
so the perturbation condition holds for $K\ge K_0$.
Claims (ii) and (iii) then follow from
Proposition~\ref{prop:pseudo-enclosure-compact} and Lemma~\ref{lem:projector-perturb}.
\end{proof}


\section{Certified resolvent bounds for discretizations (validated linear algebra)}
\label{sec:validated-linalg}

Let $A_K\in\mathbb C^{N\times N}$ represent $L_K|_{E_K}$ in a chosen basis of
$E_K=\operatorname{Ran}\pi_K$.
Throughout this section $\|\cdot\|$ denotes the Euclidean operator norm $\|\cdot\|_2$.

\subsection{Schur/SVD-based resolvent bounds with explicit defects}
\label{subsec:schur-defects}

\subsubsection*{A certified Schur build and its defects}
Let $A\in\mathbb C^{N\times N}$ denote $A_K$ for simplicity.
Compute a complex Schur decomposition (numerically):
\begin{equation}\label{eq:schur-approx}
AQ = QT + R,
\end{equation}
where $Q\in\mathbb C^{N\times N}$ is approximately unitary, $T\in\mathbb C^{N\times N}$ is upper triangular,
and $R$ is the Schur residual. Define the defects
\begin{equation}\label{eq:schur-defects}
\delta := \|I-Q^\ast Q\|,
\qquad
r_{\mathrm{sch}} := \|AQ-QT\|=\|R\|,
\qquad
C_A := \|A\|.
\end{equation}

Define the \emph{proxy} matrix
\begin{equation}\label{eq:S0-def}
S_0 := QTQ^\ast,
\qquad
E := A-S_0.
\end{equation}
\begin{lemma}[{cf.~\cite{BlumenthalNisoliTaylorCrush2025}}]\label{lem:schur-defect-to-E}
With $S_0$ and $E$ defined by~\eqref{eq:S0-def}, we have
\begin{equation}\label{eq:E-bound}
\|E\|
\ \le\
r_{\mathrm{sch}}\sqrt{1+\delta}\ +\ C_A\,\delta.
\end{equation}
Moreover, if $\delta<1$ then $Q$ is invertible and
\begin{equation}\label{eq:Q-cond}
\|Q\|\le\sqrt{1+\delta},
\qquad
\|Q^{-1}\|\le (1-\delta)^{-1/2},
\qquad
\kappa(Q):=\|Q\|\,\|Q^{-1}\|\le \sqrt{\frac{1+\delta}{1-\delta}}.
\end{equation}
\end{lemma}

\begin{proof}
Start from $E=A-QTQ^\ast$ and add/subtract $AQQ^\ast$:
\[
E
=
(AQ-QT)Q^\ast\ +\ A(I-QQ^\ast).
\]
Taking norms and using $\|Q\|\le\sqrt{1+\delta}$ (from $(1-\delta)I\preceq Q^\ast Q\preceq (1+\delta)I$)
and $\|I-QQ^\ast\|=\|I-Q^\ast Q\|=\delta$ gives~\eqref{eq:E-bound}.
The bounds~\eqref{eq:Q-cond} follow from the same Loewner inequality.
\end{proof}

\subsubsection*{Resolvent bounds via a Neumann argument}

\begin{lemma}\label{lem:resolvent-bridge-S0}
Fix $z\in\mathbb C$. If
\[
\alpha(z):=\|(zI-S_0)^{-1}\|\,\|E\|<1,
\]
then $z\in\rho(A)$ and
\begin{equation}\label{eq:resolvent-A-via-S0}
\|(zI-A)^{-1}\|
\ \le\
\frac{\|(zI-S_0)^{-1}\|}{1-\alpha(z)}.
\end{equation}
\end{lemma}

\begin{proof}
Factor $zI-A = (I-(zI-S_0)^{-1}E)\,(zI-S_0)$ and invert via Neumann series.
\end{proof}

\begin{lemma}\label{lem:resolvent-S0-via-T}
Assume $\delta=\|I-Q^\ast Q\|<1$ and let $\kappa(Q)$ be as in~\eqref{eq:Q-cond}.
Then for every $z\in\mathbb C$,
\begin{equation}\label{eq:S0-resolvent-bound}
\|(zI-S_0)^{-1}\|
\ \le\
\kappa(Q)^2\ \|(zI-T)^{-1}\|.
\end{equation}
\end{lemma}

\begin{proof}
Let $Q=UH$ be the polar decomposition. Then $S_0=U(HTH)U^\ast$, so
$(zI-S_0)^{-1}=U\,H^{-1}(zI-T)^{-1}H^{-1}\,U^\ast$ and
$\|(zI-S_0)^{-1}\|\le \|H^{-1}\|^2\,\|(zI-T)^{-1}\|\le \kappa(Q)^2\,\|(zI-T)^{-1}\|$.
\end{proof}

\begin{theorem}\label{thm:schur-to-resolvent-generic}
Let $A\in\mathbb C^{N\times N}$ and suppose~\eqref{eq:schur-approx} holds with defects
$\delta,r_{\mathrm{sch}},C_A$ as in~\eqref{eq:schur-defects}.
Assume $\delta<1$ and define $S_0,E$ by~\eqref{eq:S0-def}. Then
\[
\|E\|\ \le\ r_{\mathrm{sch}}\sqrt{1+\delta}+C_A\,\delta,
\qquad
\|(zI-S_0)^{-1}\|\ \le\ \kappa(Q)^2\,\|(zI-T)^{-1}\|.
\]
Consequently, if
\[
\beta(z):=\kappa(Q)^2\,\|(zI-T)^{-1}\|\,\|E\|<1,
\]
then $z\in\rho(A)$ and
\begin{equation}\label{eq:final-schur-resolvent}
\|(zI-A)^{-1}\|
\ \le\
\frac{\kappa(Q)^2\,\|(zI-T)^{-1}\|}{1-\beta(z)}.
\end{equation}
\end{theorem}

\begin{proof}
Combine Lemma~\ref{lem:schur-defect-to-E}, Lemma~\ref{lem:resolvent-S0-via-T}, and Lemma~\ref{lem:resolvent-bridge-S0}.
\end{proof}

\subsection{Triangular resolvent bounds: SVD and block splitting}
\label{subsec:triangular-resolvent}

Since $\|(zI-T)^{-1}\| = 1/\sigma_{\min}(zI-T)$, it suffices to produce certified lower bounds
for $\sigma_{\min}(zI-T)$.
The following result from validated linear algebra provides such bounds.

\begin{theorem}[{\cite[Thm.~3.1]{Rump2011}}]\label{thm:rump-svd}
Let $A,U,V\in\mathbb K^{n\times n}$ with
$\|I-U^HU\|\le\alpha<1$ and $\|I-V^HV\|\le\beta<1$.
Write $U^HAV=D+E$ with $D$ diagonal.
Then there exists a permutation $\nu$ such that
\[
  \frac{|D_{ii}|-\|E\|}{\sqrt{(1+\alpha)(1+\beta)}}
  \;\le\;\sigma_{\nu(i)}(A)
  \;\le\;\frac{|D_{ii}|+\|E\|}{\sqrt{(1-\alpha)(1-\beta)}}.
\]
\end{theorem}

In practice we apply Theorem~\ref{thm:rump-svd} to $A=zI-T$;
see~\cite{Miyajima2014} for extensions
and~\cite{BlumenthalNisoliTaylorCrush2025} for the implementation.

\subsubsection*{Block splitting near a target eigenvalue}
When $z$ is close to a cluster of diagonal entries of $T$, computing $\sigma_{\min}(zI-T)$ at high precision
for the full $N\times N$ matrix can be wasteful.
Instead, we isolate the difficult part in a small leading block.

Let $T$ be Schur-ordered so that the $k$ diagonal entries closest to a target $\lambda_j$ appear first, and
partition
\[
T=
\begin{pmatrix}
T_{11} & T_{12}\\
0 & T_{22}
\end{pmatrix}.
\]

\begin{lemma}\label{lem:block-weyl}
For every $z\in\mathbb C$,
\[
\sigma_{\min}(zI-T)
\ \ge\
\min\bigl(\sigma_{\min}(zI-T_{11}),\ \sigma_{\min}(zI-T_{22})\bigr)\ -\ \|T_{12}\|.
\]
\end{lemma}

\begin{proof}
Write $zI-T = M(z)+E$ with
\[
M(z):=\begin{pmatrix}zI-T_{11}&0\\0&zI-T_{22}\end{pmatrix},
\qquad
E:=\begin{pmatrix}0&-T_{12}\\0&0\end{pmatrix}.
\]
Then $\sigma_{\min}(M(z))=\min(\sigma_{\min}(zI-T_{11}),\sigma_{\min}(zI-T_{22}))$ and $\|E\|=\|T_{12}\|$.
Apply Weyl: $\sigma_{\min}(M(z)+E)\ge \sigma_{\min}(M(z))-\|E\|$.
\end{proof}

\subsubsection*{Sampling a contour and Lipschitz propagation}

\begin{lemma}\label{lem:sigmin-lipschitz-2}
For any fixed matrix $A$ and any $z,w\in\mathbb C$,
\[
\bigl|\sigma_{\min}(zI-A)-\sigma_{\min}(wI-A)\bigr|\le |z-w|.
\]
\end{lemma}

\begin{proof}
Since $(zI-A)-(wI-A)=(z-w)I$ and $\|(z-w)I\|=|z-w|$, Weyl's inequality gives
$|\sigma_{\min}(zI-A)-\sigma_{\min}(wI-A)|\le \|(z-w)I\|=|z-w|$.
\end{proof}

For a circle $\Gamma=\{\,\lambda+\rho e^{it}\,\}$ sampled at $m$ equally spaced points,
every point of the circle lies within chord distance
\[
\Delta_m:=2\rho\sin\!\Bigl(\frac{\pi}{2m}\Bigr)
\]
of some sample point.

\begin{proposition}\label{prop:triangular-circle}
Let $T$ be upper triangular and let $\Gamma=\{z:\ |z-\lambda|=\rho\}$.
Fix sample points $z_\ell=\lambda+\rho e^{2\pi i\ell/m}$, $\ell=0,\dots,m-1$, and let
\[
s_\ell \ \le\ \sigma_{\min}(z_\ell I-T)
\qquad (\ell=0,\dots,m-1)
\]
be certified lower bounds. Set
\[
s_*:=\min_{0\le\ell<m} s_\ell\ -\ \Delta_m,
\qquad
\Delta_m:=2\rho\sin\!\Bigl(\frac{\pi}{2m}\Bigr).
\]
If $s_*>0$, then $\Gamma\subset\rho(T)$ and
\[
\sup_{z\in\Gamma}\ \|(zI-T)^{-1}\|\ \le\ \frac{1}{s_*}.
\]
\end{proposition}

\begin{proof}
For any $z\in\Gamma$ choose $\ell$ with $|z-z_\ell|\le\Delta_m$ and apply Lemma~\ref{lem:sigmin-lipschitz-2}:
\[
\sigma_{\min}(zI-T)\ge \sigma_{\min}(z_\ell I-T)-|z-z_\ell|\ge s_\ell-\Delta_m\ge s_*.
\]
Thus $\sigma_{\min}(zI-T)\ge s_*>0$ for all $z\in\Gamma$, hence all $zI-T$ are invertible and
$\|(zI-T)^{-1}\|=1/\sigma_{\min}(zI-T)\le 1/s_*$.
\end{proof}

\begin{remark}\label{rem:adaptive-sampling}
In~\cite{BlumenthalNisoliTaylorCrush2025} the equispaced sampling of
Proposition~\ref{prop:triangular-circle} is replaced by an adaptive method
that refines the mesh where $\sigma_{\min}(z_\ell I-T)$ is small, reducing
the number of certified SVD evaluations while maintaining the same
Lipschitz-propagation guarantee.
\end{remark}

\begin{remark}\label{rem:block-splitting-enters}
In Proposition~\ref{prop:triangular-circle}, the bounds $s_\ell$ can be produced using
Lemma~\ref{lem:block-weyl}:
compute a certified $c_{12}\ge\|T_{12}\|$ once, then at each sample point bound
$\sigma_{\min}(z_\ell I-T_{11})$ (high precision, small $k$) and $\sigma_{\min}(z_\ell I-T_{22})$
(fast, since $T_{22}$ is well separated), and combine
\[
s_\ell := \min\bigl(s^{(11)}_\ell,\ s^{(22)}_\ell\bigr)-c_{12}.
\]
\end{remark}

\subsection{Contour evaluation and certified suprema}
\label{subsec:contour-suprema}

We summarize the certification workflow for producing $M_\Gamma(A):=\sup_{z\in\Gamma}\|(zI-A)^{-1}\|$
on a contour $\Gamma=\partial U$.

\begin{enumerate}
\item Compute $(Q,T)$ and the defects $\delta,r_{\mathrm{sch}},C_A$ in~\eqref{eq:schur-defects}.
Form the bound on $\|E\|$ from Lemma~\ref{lem:schur-defect-to-E} and the conditioning factor $\kappa(Q)$.

\item Choose a sampling of $\Gamma$ (e.g.\ circles around selected diagonal entries of $T$).
At sample points $z_\ell$, produce certified lower bounds for $\sigma_{\min}(z_\ell I-T)$
(either directly by SVD, or by block splitting as in Remark~\ref{rem:block-splitting-enters}).
Use Proposition~\ref{prop:triangular-circle} (or the obvious analogue for a general piecewise $C^1$ contour)
to obtain a bound
\[
M_\Gamma(T)\ :=\ \sup_{z\in\Gamma}\|(zI-T)^{-1}\|.
\]

\item Define
\[
\beta_\Gamma := \kappa(Q)^2\,M_\Gamma(T)\,\|E\|.
\]
If $\beta_\Gamma<1$, then Theorem~\ref{thm:schur-to-resolvent-generic} gives $\Gamma\subset\rho(A)$ and
\[
M_\Gamma(A)\ \le\ \frac{\kappa(Q)^2\,M_\Gamma(T)}{1-\beta_\Gamma}.
\]
\end{enumerate}

If $\varepsilon_K\,M_\Gamma(A)<1$, then $\Gamma$ is a certified exclusion curve for the original operator $L$.

\subsection{Projector error decomposition}
\label{subsec:proj-split-schur-ord}

Fix a bounded domain $U\subset\mathbb C$ whose boundary $\Gamma:=\partial U$ is a piecewise $C^1$ Jordan curve.
For any matrix $M$, write the Riesz projector
\[
P_M(U):=\frac{1}{2\pi i}\int_{\Gamma} (zI-M)^{-1}\,dz.
\]

Let $A\in\mathbb C^{N\times N}$ and suppose we have a (numerical) Schur pair $(Q,T)$ with
\[
AQ=QT+R_{\mathrm{sch}},\qquad r_{\mathrm{sch}}:=\|R_{\mathrm{sch}}\|_2,
\]
where $Q$ is unitary (or already certified as nearly unitary).
Define the Schur-basis representation
\[
\widehat A := Q^\ast A Q = T + E_{\mathrm{sch}},
\qquad
E_{\mathrm{sch}}:=Q^\ast R_{\mathrm{sch}},
\qquad
\|E_{\mathrm{sch}}\|_2=r_{\mathrm{sch}}.
\]

Next, suppose a reordering routine produces $\widehat U$ and $\tilde T$ such that
\[
T\,\widehat U = \widehat U\,\tilde T + R_{\mathrm{ord}},
\qquad
\delta_{\mathrm{ord}}:=\|R_{\mathrm{ord}}\|_2,
\]
and let the orthogonality defect be
\[
\Delta_U:=I-\widehat U^\ast \widehat U,\qquad \delta_U:=\|\Delta_U\|_2.
\]
Assume $\delta_U<1$ so that $\widehat U$ is invertible and
$\|\widehat U^{-1}\|_2\le (1-\delta_U)^{-1/2}$.

Since $\tilde T$ is expressed in the reordered basis, we compare projectors in the original Schur coordinates via
\[
\tilde P_T(U):=\widehat U\,P_{\tilde T}(U)\,\widehat U^{-1}.
\]

\begin{lemma}\label{lem:proj-triangle-split}
With the objects above,
\[
\|P_{\widehat A}(U)-\tilde P_T(U)\|
\le
\|P_{\widehat A}(U)-P_T(U)\| \;+\; \|P_T(U)-\widehat U P_{\tilde T}(U)\widehat U^{-1}\|.
\]
\end{lemma}

\begin{proof}
Insert and subtract $P_T(U)$ and use the triangle inequality.
\end{proof}

The first term in the triangle inequality is controlled by the Schur residual. Let
\[
M_T:=\sup_{z\in\Gamma}\|(zI-T)^{-1}\|_2.
\]
If
\[
\eta_{\mathrm{sch}}:=\|E_{\mathrm{sch}}\|_2\,M_T = r_{\mathrm{sch}}\,M_T < 1,
\]
then $\Gamma\subset\rho(\widehat A)$ and the resolvent perturbation bound gives
\[
\sup_{z\in\Gamma}\|(zI-\widehat A)^{-1}-(zI-T)^{-1}\|_2
\le \frac{r_{\mathrm{sch}}\,M_T^2}{1-\eta_{\mathrm{sch}}}.
\]
Consequently,
\begin{equation}\label{eq:proj-bound-schur}
\|P_{\widehat A}(U)-P_T(U)\|
\le
\frac{\ell(\Gamma)}{2\pi}\cdot \frac{r_{\mathrm{sch}}\,M_T^2}{1-\eta_{\mathrm{sch}}},
\end{equation}
where $\ell(\Gamma)$ denotes the length of $\Gamma$.

The second term is controlled by the reordering defect.
From $T\widehat U=\widehat U\tilde T+R_{\mathrm{ord}}$ we have
\[
T = \widehat U \tilde T\,\widehat U^{-1} + E_{\mathrm{ord}},
\qquad
E_{\mathrm{ord}}:=R_{\mathrm{ord}}\,\widehat U^{-1},
\qquad
\|E_{\mathrm{ord}}\|_2 \le \delta_{\mathrm{ord}}\,(1-\delta_U)^{-1/2}.
\]
If
\[
\eta_{\mathrm{ord}}:=\|E_{\mathrm{ord}}\|_2\,M_T < 1,
\]
then $\Gamma$ also lies in the resolvent set of $\widehat U\tilde T\widehat U^{-1}$ and
\[
\sup_{z\in\Gamma}\|(zI-T)^{-1}-(zI-\widehat U\tilde T\widehat U^{-1})^{-1}\|_2
\le \frac{\|E_{\mathrm{ord}}\|_2\,M_T^2}{1-\eta_{\mathrm{ord}}}.
\]
Using similarity invariance of the Riesz projector in exact arithmetic,
$P_{\widehat U\tilde T\widehat U^{-1}}(U)=\widehat U P_{\tilde T}(U)\widehat U^{-1}$, we obtain
\begin{equation}\label{eq:proj-bound-ord}
\|P_T(U)-\widehat U P_{\tilde T}(U)\widehat U^{-1}\|
\le
\frac{\ell(\Gamma)}{2\pi}\cdot \frac{\|E_{\mathrm{ord}}\|_2\,M_T^2}{1-\eta_{\mathrm{ord}}}.
\end{equation}

Under $\eta_{\mathrm{sch}}<1$ and $\eta_{\mathrm{ord}}<1$,
Lemma~\ref{lem:proj-triangle-split} together with \eqref{eq:proj-bound-schur}--\eqref{eq:proj-bound-ord}
yields an explicit certified estimate for $\|P_{\widehat A}(U)-\tilde P_T(U)\|$.
Finally, since $P_A(U)=Q\,P_{\widehat A}(U)\,Q^\ast$ and $Q$ is unitary, this also bounds the
projector error in the original coordinates of $A$.

\section{Hardy-space framework for the Gauss--Kuzmin--Wirsing operator}
\label{sec:gkw-hardy}

\subsection{Hardy spaces and branch geometry}
\label{subsec:hardy-geometry}

For $r>0$ let
\[
D_r:=\{\,w\in\mathbb C:\ |w-1|<r\,\}
\]
and let $H^2(D_r)$ be the Hardy space on $D_r$, realized as holomorphic functions on $D_r$ with square
summable Taylor coefficients at $1$.

\begin{definition}\label{def:H2-Dr}
For $r>0$, define
\[
H^2(D_r):=\left\{\, f(w)=\sum_{k\ge 0}c_k(w-1)^k \ \text{holomorphic on }D_r:\ 
\|f\|_{H^2(D_r)}^2:=\sum_{k\ge 0}|c_k|^2 r^{2k}<\infty \right\}.
\]
With the inner product $\langle f,g\rangle_{H^2(D_r)}:=\sum_{k\ge 0} c_k\overline{d_k}\,r^{2k}$,
$H^2(D_r)$ is a Hilbert space. If $0<r<R$, restriction defines a continuous embedding
\[
J_{R\to r}:H^2(D_R)\hookrightarrow H^2(D_r),
\qquad (J_{R\to r}f)(w)=f(w)\big|_{D_r},
\]
with operator norm $\|J_{R\to r}\|=1$.
\end{definition}

The Gauss inverse branches are
\[
\tau_n(w)=\frac{1}{w+n},\qquad n\ge 1.
\]
They map discs $D_{3/2}$ strictly inside $D_1$, which is the key geometric fact behind the
strong-weak smoothing that later yields compactness and explicit approximation bounds.

\begin{lemma}\label{lem:branch-geometry}
For every $n\ge 1$ one has
\[
\tau_n(D_{3/2})\subset D_1.
\]
\end{lemma}

\begin{proof}
If $w\in D_{3/2}$, then $\Re w>-\tfrac12$. Compute
\[
|\tau_n(w)-1|
=\left|\frac{1}{w+n}-1\right|
=\frac{|w+n-1|}{|w+n|}.
\]
Moreover
\[
|w+n|^2-|w+n-1|^2
=(w+n)(\overline w+n)-(w+n-1)(\overline w+n-1)
=2\Re w+(2n-1).
\]
Since $\Re w>-\tfrac12$, we get $2\Re w+(2n-1)>2n-2\ge 0$. Hence $|w+n|>|w+n-1|$ and therefore
$|\tau_n(w)-1|<1$, i.e.\ $\tau_n(w)\in D_1$.
\end{proof}

\begin{lemma}\label{lem:weight-bound}
For every $n\ge 1$ and every $w\in D_{3/2}$,
\[
|(w+n)^{-2}|\ \le\ (n-\tfrac12)^{-2}.
\]
\end{lemma}

\begin{proof}
If $w\in D_{3/2}$ then $\Re w> -\tfrac12$ and hence $|w+n|> n-\tfrac12$, which implies the claim.
\end{proof}

\subsection{Boundedness and compactness}
\label{subsec:boundedness-compactness}

Define the (un-normalized) Gauss--Kuzmin--Wirsing transfer operator by
\begin{equation}\label{eq:def-GKW}
(\mathcal L f)(w)
:=\sum_{n\ge 1}\frac{1}{(w+n)^2}\, f\!\left(\frac{1}{w+n}\right)
=\sum_{n\ge 1} \psi_n(w)\,(f\circ\tau_n)(w),
\qquad
\psi_n(w):=(w+n)^{-2}.
\end{equation}
The geometric inclusion of Lemma~\ref{lem:branch-geometry} ensures that if $f$ is holomorphic on $D_1$,
then $f\circ\tau_n$ is holomorphic on $D_{3/2}$, so each summand in \eqref{eq:def-GKW} is holomorphic on
$D_{3/2}$.

\begin{proposition}\label{prop:L-bounded}
The operator $\mathcal L$ defined by \eqref{eq:def-GKW} extends to a bounded operator
\[
\mathcal L:H^2(D_1)\longrightarrow H^2(D_{3/2}),
\]
and
\[
\|\mathcal L\|_{H^2(D_1)\to H^2(D_{3/2})}
\ \le\ 
\sum_{n\ge 1}\frac{\|C_{\tau_n}\|_{H^2(D_1)\to H^2(D_{3/2})}}{(n-\tfrac12)^2}
=:C_2,
\]
where $C_{\tau_n}:f\mapsto f\circ\tau_n$ is the composition operator.
In particular, using the estimate $\|C_{\tau_n}\|\le \sqrt{2n+1}$ (proved below),
\[
\|\mathcal L\|_{H^2(D_1)\to H^2(D_{3/2})}
\ \le\
\sum_{n\ge 1}\frac{\sqrt{2n+1}}{(n-\tfrac12)^2}
\ \le\
\sqrt6\,\bigl(2^{3/2}-1\bigr)\,\zeta\!\left(\tfrac32\right)
=:C_2.
\]
\end{proposition}

\begin{proof}
Fix $f\in H^2(D_1)$. For each $n\ge 1$, the map $w\mapsto (f\circ\tau_n)(w)$ is holomorphic on $D_{3/2}$
by Lemma~\ref{lem:branch-geometry}, and multiplication by the bounded holomorphic weight $\psi_n$ preserves
holomorphy. Using the operator norm on $H^2(D_{3/2})$ and Lemma~\ref{lem:weight-bound},
\[
\|\psi_n\,(f\circ\tau_n)\|_{H^2(D_{3/2})}
\le \sup_{w\in D_{3/2}}|\psi_n(w)|\ \|f\circ\tau_n\|_{H^2(D_{3/2})}
\le (n-\tfrac12)^{-2}\,\|C_{\tau_n}\|\,\|f\|_{H^2(D_1)}.
\]
Summing in $n$ yields
\[
\|\mathcal L f\|_{H^2(D_{3/2})}
\le \left(\sum_{n\ge 1}\frac{\|C_{\tau_n}\|}{(n-\tfrac12)^2}\right)\|f\|_{H^2(D_1)},
\]
so $\mathcal L$ is bounded and the first inequality follows.

For the explicit constant, combine $\|C_{\tau_n}\|\le \sqrt{2n+1}$ (Lemma~\ref{lem:comp-norm} below)
with the comparison $\sqrt{2n+1}\le \sqrt6\,(n-\tfrac12)^{1/2}$ and the identity
$\sum_{n\ge 1}(n-\tfrac12)^{-3/2}=(2^{3/2}-1)\zeta(3/2)$.
\end{proof}

\begin{lemma}\label{lem:comp-norm}
For every $n\ge 1$, the composition operator
\[
C_{\tau_n}:H^2(D_1)\to H^2(D_{3/2}),\qquad C_{\tau_n}f=f\circ\tau_n,
\]
satisfies
\[
\|C_{\tau_n}\|_{H^2(D_1)\to H^2(D_{3/2})}\ \le\ \sqrt{2n+1}.
\]
\end{lemma}

\begin{proof}
Let $T_r(\zeta)=1+r\zeta$ map $\mathbb D$ biholomorphically onto $D_r$, and define the unitary
$U_r:H^2(D_r)\to H^2(\mathbb D)$ by $(U_r f)(\zeta)=f(T_r(\zeta))$.
Then $U_{3/2}\,C_{\tau_n}\,U_1^{-1}=C_{\phi_n}$ with
\[
\phi_n:=T_1^{-1}\circ\tau_n\circ T_{3/2}.
\]
A direct computation gives $\phi_n(0)=-\frac{n}{n+1}$.
By the standard Hardy-space estimate for composition operators on $H^2(\mathbb D)$ \cite{CowenMaccluer2019},
\[
\|C_{\phi}\|_{H^2(\mathbb D)}\le\left(\frac{1+|\phi(0)|}{1-|\phi(0)|}\right)^{1/2},
\]
hence
\[
\|C_{\tau_n}\|=\|C_{\phi_n}\|
\le\left(\frac{1+\frac{n}{n+1}}{1-\frac{n}{n+1}}\right)^{1/2}
=\sqrt{2n+1}.
\]
\end{proof}

Let $S:=J_{3/2\to 1}:H^2(D_{3/2})\to H^2(D_1)$ denote the restriction map (Definition~\ref{def:H2-Dr}).
Define
\[
L_{1}:=S\,\mathcal L:H^2(D_1)\to H^2(D_1),
\qquad
L_{3/2}:=\mathcal L\,S:H^2(D_{3/2})\to H^2(D_{3/2}).
\]

\begin{remark}\label{rem:notation-L}
Throughout this section we use the subscripted notation $L_{1}$ and $L_{3/2}$ to
distinguish the two realizations.
From Section~\ref{sec:gkw-certified} onward,
we write simply $L$ for the operator $L_{1}=S\,\mathcal L$ on $H^2(D_1)$
and $L_K$ for its finite-rank truncation.
\end{remark}

\begin{proposition}\label{prop:compactness-L1-L32}
The operators $L_{1}$ and $L_{3/2}$ are compact.
\end{proposition}

\begin{proof}
We will show that $L_{1}$ is the operator-norm limit of finite-rank truncations, and the argument for
$L_{3/2}$ is identical.

Let $\Pi_K:H^2(D_{3/2})\to H^2(D_{3/2})$ be the degree-$K$ Taylor truncation at $1$:
\[
\Pi_K\!\left(\sum_{m\ge 0} a_m(w-1)^m\right)=\sum_{m=0}^{K} a_m(w-1)^m.
\]
Then $\Pi_K$ has rank $K+1$. Define $(L_{1})_K:=\Pi_K\,S\,\mathcal L$. Since $\Pi_K$ is finite rank,
so is $(L_{1})_K$. In Subsection~\ref{subsec:truncation} below we prove the tail estimate
$\|L_{1}-(L_{1})_K\|\le C_2(2/3)^{K+1}\to 0$, which implies that $L_{1}$ is compact.
The same construction on the domain side yields compactness of $L_{3/2}$.
\end{proof}

\subsection{Explicit truncation bounds and construction of $L_N$ with $\|L-L_N\|\le\varepsilon_N$}
\label{subsec:truncation}

\begin{definition}\label{def:PiK}
For $r>0$ and $K\ge 0$, let $\Pi_K:H^2(D_r)\to H^2(D_r)$ be the degree-$K$ truncation at $1$:
\[
\Pi_K f(w):=\sum_{m=0}^{K} c_m(w-1)^m,\qquad
f(w)=\sum_{m\ge 0} c_m(w-1)^m.
\]
Then $\Pi_K$ has rank $K+1$ and operator norm $\|\Pi_K\|=1$.
\end{definition}

\begin{lemma}\label{lem:H2-tail}
Let $0<r<R$ and let $J_{R\to r}:H^2(D_R)\to H^2(D_r)$ be restriction. Then for every $K\ge 0$ and every
$f\in H^2(D_R)$,
\[
\|\,J_{R\to r}f-\Pi_K(J_{R\to r}f)\,\|_{H^2(D_r)}
\ \le\
\left(\frac{r}{R}\right)^{K+1}\,\|f\|_{H^2(D_R)}.
\]
In particular, for $(r,R)=(1,3/2)$,
\[
\|\,J_{3/2\to 1}f-\Pi_K(J_{3/2\to 1}f)\,\|_{H^2(D_1)}
\ \le\
\left(\frac{2}{3}\right)^{K+1}\,\|f\|_{H^2(D_{3/2})}.
\]
\end{lemma}

\begin{proof}
Write $f(w)=\sum_{m\ge 0}c_m(w-1)^m$ on $D_R$. Then
\[
\|J_{R\to r}f-\Pi_K(J_{R\to r}f)\|_{H^2(D_r)}^2
=\sum_{m>K}|c_m|^2 r^{2m}
=\sum_{m>K}|c_m|^2 R^{2m}\left(\frac{r}{R}\right)^{2m}.
\]
Since $m>K$ implies $(r/R)^{2m}\le (r/R)^{2(K+1)}$, we obtain
\[
\|J_{R\to r}f-\Pi_K(J_{R\to r}f)\|_{H^2(D_r)}^2
\le \left(\frac{r}{R}\right)^{2(K+1)} \sum_{m>K}|c_m|^2 R^{2m}
\le \left(\frac{r}{R}\right)^{2(K+1)} \|f\|_{H^2(D_R)}^2,
\]
and taking square roots gives the claim.
\end{proof}

\begin{theorem}\label{thm:one-sided-approxs}
Let $\mathcal L:H^2(D_1)\to H^2(D_{3/2})$ be the GKW operator from \eqref{eq:def-GKW}, and set
\[
L_{1}:=S\,\mathcal L:H^2(D_1)\to H^2(D_1),
\qquad
L_{3/2}:=\mathcal L\,S:H^2(D_{3/2})\to H^2(D_{3/2}).
\]
For $K,k\ge 0$ define the finite-rank truncations
\[
(L_{1})_{K}:=\Pi_{K}\,S\,\mathcal L,
\qquad
(L_{3/2})_{k}:=\mathcal L\,\Pi_{k}\,S.
\]
Then
\[
\|L_{1}-(L_{1})_{K}\|_{H^2(D_1)\to H^2(D_1)}
\ \le\ C_2\left(\frac{2}{3}\right)^{K+1},
\qquad
\|L_{3/2}-(L_{3/2})_{k}\|_{H^2(D_{3/2})\to H^2(D_{3/2})}
\ \le\ C_2\left(\frac{2}{3}\right)^{k+1},
\]
where $C_2:=\|\mathcal L\|_{H^2(D_1)\to H^2(D_{3/2})}$ (cf.\ Proposition~\ref{prop:L-bounded}).
\end{theorem}

\begin{proof}
(i) Let $f\in H^2(D_1)$ and set $g:=\mathcal L f\in H^2(D_{3/2})$. Then $L_{1}f=Sg=J_{3/2\to 1}g$.
By Lemma~\ref{lem:H2-tail} with $(r,R)=(1,3/2)$,
\[
\|L_{1}f-(L_{1})_{K}f\|_{H^2(D_1)}
=\|\,Sg-\Pi_{K}Sg\,\|_{H^2(D_1)}
\le \left(\frac{2}{3}\right)^{K+1}\|g\|_{H^2(D_{3/2})}
\le C_2\left(\frac{2}{3}\right)^{K+1}\|f\|_{H^2(D_1)}.
\]
Taking the supremum over $\|f\|_{H^2(D_1)}=1$ gives the first bound.

(ii) Let $F\in H^2(D_{3/2})$. Then
\[
\|L_{3/2}F-(L_{3/2})_{k}F\|_{H^2(D_{3/2})}
=\|\mathcal L(SF-\Pi_{k}SF)\|_{H^2(D_{3/2})}
\le C_2\,\|SF-\Pi_{k}SF\|_{H^2(D_1)}.
\]
Apply Lemma~\ref{lem:H2-tail} with $(r,R)=(1,3/2)$ to $F$ to get
$\|SF-\Pi_{k}SF\|_{H^2(D_1)}\le (2/3)^{k+1}\|F\|_{H^2(D_{3/2})}$, yielding the second bound.
\end{proof}

\begin{corollary}
\label{cor:eig-decay-traceclass}
The approximation numbers of $L_{1}$ and $L_{3/2}$ satisfy
\[
a_m(L_{1})\ \le\ C_2\Bigl(\tfrac{2}{3}\Bigr)^{m-1},
\qquad
a_m(L_{3/2})\ \le\ C_2\Bigl(\tfrac{2}{3}\Bigr)^{m-1}
\qquad(m\ge1).
\]
Hence both operators are trace class, with trace norm ($\|\cdot\|_{S_1}:=\sum_{m\ge1}s_m$)
\[
\|L_{1}\|_{S_1}\ \le\ 3\,C_2,
\qquad
\|L_{3/2}\|_{S_1}\ \le\ 3\,C_2,
\]
and their eigenvalues $\{\lambda_m\}_{m\ge1}$ (ordered by decreasing modulus, with multiplicity) satisfy
\[
|\lambda_m(L_{1})|,\ |\lambda_m(L_{3/2})|
\ \le\ C_2\Bigl(\tfrac{2}{3}\Bigr)^{m-1},
\qquad
\sum_{m\ge1}|\lambda_m|<\infty.
\]
\end{corollary}

\begin{proof}
For $m\ge 2$, choose $K=m-2$ in Theorem~\ref{thm:one-sided-approxs}; then
$(L_{1})_{K}$ has rank $\le m-1$, so
\[
a_m(L_{1})\ \le\ \|L_{1}-(L_{1})_{m-2}\|
\ \le\ C_2\Bigl(\tfrac{2}{3}\Bigr)^{m-1}.
\]
For $m=1$ we have $a_1(L_{1})\le\|L_{1}\|\le C_2$.
On $H^2(D_r)$ the approximation numbers coincide with the singular values, and
Weyl's inequality gives $|\lambda_m|\le s_m=a_m$
\cite[Thm.~1.33]{Simon05}.
Finally, $\sum_{m\ge1}s_m(L_{1})\le C_2\sum_{m\ge1}(2/3)^{m-1}=3C_2<\infty$,
so $L_{1}$ is trace class \cite[Thm.~1.35]{Simon05}, and $\sum|\lambda_m|\le\sum s_m<\infty$.
The argument for $L_{3/2}$ is identical.
\end{proof}

\section{Certified spectral enclosures for the Gauss--Kuzmin--Wirsing operator}
\label{sec:gkw-certified}

Throughout, the cross-space operator norm
\[
C_2:=\|\mathcal L\|_{H^2(D_1)\to H^2(D_{3/2})}
\]
enters only through the explicit truncation budget
\[
\varepsilon_K = C_2\Big(\frac23\Big)^{K+1},
\]
so improving the numerical bound for \(C_2\) immediately sharpens all a posteriori certificates.

\subsection{Refined bounds for \(C_2\) and truncation budgets}

\begin{remark}\label{rem:C2-splitting}
Write
\[
\frac{\sqrt{2n+1}}{(n-\tfrac12)^2}
=\sqrt{\frac{2n+1}{\,n-\tfrac12\,}}\,(n-\tfrac12)^{-3/2}.
\]
For any $N\ge 1$ we have
\[
\sup_{n\ge N}\frac{2n+1}{n-\tfrac12}
=2+\frac{2}{N-\tfrac12},
\]
hence, splitting the series at $N$ yields
\[
C_2
=
\|\mathcal L\|_{H^2(D_1)\to H^2(D_{3/2})}
\le
\sum_{n=1}^{N-1}\frac{\sqrt{2n+1}}{(n-\tfrac12)^2}
\;+\;
\sqrt{\,2+\frac{2}{N-\tfrac12}\,}\,
\sum_{n\ge N}(n-\tfrac12)^{-3/2}.
\]
The tail can be expressed via the Hurwitz zeta function:
\[
\sum_{n\ge N}(n-\tfrac12)^{-3/2}
=\zeta\!\left(\tfrac32,\,N-\tfrac12\right).
\]
Table~\ref{tab:C2-bounds} reports rigorous upper bounds for $C_2$ obtained from this
splitting strategy, computed with \texttt{Arb}; the maximum enclosure radius is bounded by
$5\times 10^{-34}$. (The case $N=1$ coincides with the base estimate from
Proposition~\ref{prop:L-bounded}.)
\end{remark}

\begin{table}[htbp]
\centering
\begin{tabular}{r r}
\hline
$N$ & $C_2(N)$ \\ \hline
1 & 11.7000807086 \\
2 & 10.4849507042 \\
3 & 10.2709840576 \\
4 & 10.1907129342 \\
5 & 10.1506945125 \\
10 & 10.0893118365 \\
100 & 10.0590444185 \\
1000 & 10.0581271764 \\
10000 & 10.0580982915 \\
\hline
\end{tabular}
\caption{Refined upper bounds for \(C_2\) as a function of the splitting \(N\).}
\label{tab:C2-bounds}
\end{table}

Once a bound for \(C_2\) is fixed, all truncation budgets used below are explicit and of the form
\[
\varepsilon_K = C_2\Big(\frac23\Big)^{K+1},
\]
as in Theorem~\ref{thm:one-sided-approxs}. This is the only infinite-dimensional input
to the certification step.
\begin{remark}
Note that as soon as we have a validated bound on $\|L_K\|$ we can update $C_2$.
In particular, if we are able to rigorously prove that $\|L_K\|\leq \tilde{C}_2$, we have that
\[
\|L\|\leq \tilde{C}_2+\|L\|\Big(\frac23\Big)^{K+1}
\]
and we get the finer bound
\[
\|L\|\leq \frac{\tilde{C}_2}{1-\Big(\frac23\Big)^{K+1}}
\]
\end{remark}

\subsection{Discretization and matrix assembly}
\label{subsec:gkw-discretization}

We discretize \(L=S\mathcal L\) on \(H^2(D_1)\) by truncating to the shifted monomial subspace
\[
\mathcal V_{K+1}:=\mathrm{span}\{(w-1)^j\}_{j=0}^{K}.
\]
Let \(\Pi_K\) denote the corresponding truncation operator (degree \(\le K\)).
The finite-rank discretization is then represented, in this basis, by a matrix
\(A_K\in\mathbb C^{(K+1)\times (K+1)}\).

We assemble \(A_K\) using exact coefficient formulas expressed
in terms of Hurwitz zeta values, evaluated in ball arithmetic (\texttt{Arb}).
Concretely, for each column \(k\) we expand the image of the basis element
\(\phi_k(w)=(w-1)^k\) in a Taylor series at \(w=1\),
\[
(L\phi_k)(w)=\sum_{\ell\ge 0} (A_K)_{\ell,k}\,(w-1)^\ell,
\qquad 0\le k\le K,
\]
and retain the coefficients \(\ell=0,\dots,K\).
The resulting matrix entries are complex balls, hence already include rigorous
enclosures for all evaluation/rounding errors at assembly time.

The operator-level truncation error satisfies
\[
\|L - L_K\|\le \varepsilon_K,
\]
with \(\varepsilon_K\) as above; this is the quantity used in the resolvent perturbation bound and
projector perturbation bounds.
The numerical assembly and certification results are presented in the following subsections.

%
%
%

\subsection{Certified resolvent bounds on contours}
\label{sec:certified-resolvent}

To lift the finite-dimensional spectral data of $L_K$ to the
infinite-dimensional operator $L : H^2(D_1) \to H^2(D_1)$,
we verify the \emph{small-gain condition}: for each eigenvalue
$\lambda_j$ and an excluding circle
$\Gamma_j = \{z : |z - \hat\lambda_j| = r_j\}$,
\begin{equation}
\label{eq:small-gain}
\alpha_j \;=\; \varepsilon_K \cdot
  \sup_{z \in \Gamma_j} \|(z I - L_K)^{-1}\| \;< \;1,
\end{equation}
where $\varepsilon_K = C_2 (2/3)^{K+1}$ is the truncation error bound
(Theorem~\ref{thm:one-sided-approxs}).
When \eqref{eq:small-gain} holds, $\Gamma_j \subset \rho(L)$ and $L$
has exactly as many eigenvalues (counted with multiplicity) inside
$\Gamma_j$ as $L_K$.

We employ a \emph{two-stage} strategy, following the
certification workflow of \S\ref{subsec:contour-suprema}.
For eigenvalues $j = 1, \ldots, 15$, the resolvent of the full matrix
$L_{48}$ is sampled at $256$ points on $\Gamma_j$ and certified
via Theorem~\ref{thm:schur-to-resolvent-generic}
(Schur defects $\to$ resolvent bound),
yielding $\alpha_j \leq 6.2 \times 10^{-1}$
($\varepsilon_{48} \approx 2.37 \times 10^{-8}$).
For $j = 16, \ldots, 50$, the eigenvalue gaps shrink below $10^{-7}$
and the resolvent of $L_{48}$ becomes too large;
instead, we certify the resolvent at $K = 256$
($\varepsilon_{256} \approx 5.59 \times 10^{-45}$)
using the \emph{block Schur method}: the certified Schur form
$L_{256} = Q T Q^*$ is partitioned around $\lambda_j$
(Lemma~\ref{lem:block-weyl}), and the
resolvent supremum on each circle $\Gamma_j$ is certified by
Proposition~\ref{prop:triangular-circle},
giving $\alpha_j \leq 3.14 \times 10^{-4}$
for all $50$ eigenvalues.

The certified finite-dimensional resolvent $\|R_{L_K}\|$ is then
lifted to the infinite-dimensional one via Neumann perturbation:
\[
\|R_L(z)\| \;\leq\; \frac{\|R_{L_K}(z)\|}{1 - \alpha_j}
\;=:\; M_{\infty,j}.
\]
The full results are given in the supplementary material
(Tables~S3 and~S4);
all certified data are also available in machine-readable form
in~\cite{NisoliGKWData2026}.
All $50$ excluding circles are certified, proving that the first $50$
eigenvalues of~$L$ are simple.

\subsection{Eigenvalue, eigenvector, and projector enclosures}
\label{sec:enclosures}

The resolvent certification also yields bounds on the Riesz projector errors.
Writing $P_j = P_L(\Gamma_j)$ and $P_{K,j} = P_{L_K}(\Gamma_j)$
for the spectral projectors of $L$ and $L_K$ associated with
the excluding circle $\Gamma_j$, the contour integral representation gives
\begin{equation}
\label{eq:proj-error}
\vartheta_j \;:=\; \|P_j - P_{K,j}\|
\;\leq\; \frac{|\Gamma_j|}{2\pi} \cdot
  \frac{M_{\infty,j}^2 \cdot \varepsilon_K}{1 - \alpha_j}.
\end{equation}
At the \emph{finite-dimensional level}, the projector error
$\|P_{L_K}(\Gamma_j) - \tilde P_T(\Gamma_j)\|$ arising from the
numerical Schur decomposition and reordering is bounded by the
decomposition of Lemma~\ref{lem:proj-triangle-split}: the Schur
residual contributes~\eqref{eq:proj-bound-schur} and the ordschur
reordering contributes~\eqref{eq:proj-bound-ord}.
At $K = 256$, $\vartheta_j$ ranges from $1.01 \times 10^{-41}$ ($j = 1$)
to $2.96 \times 10^{-3}$ ($j = 35$).  For $j \geq 36$, useful
projector bounds require higher $K$; we use $K = 1024$
(see supplementary Table~S3).

Since each $\lambda_j$ is certified simple and $\vartheta_j < 1$,
Lemma~\ref{lem:proj-controls-eval} gives the eigenvalue error bound
\begin{equation}
\label{eq:eval-from-proj}
|\lambda_j - \hat\lambda_j|
\;\leq\;
\frac{\varepsilon_K(1 + \vartheta_j) + 2C\,\vartheta_j}{1 - \vartheta_j},
\end{equation}
where $C = \|L_K\|$ is a certified upper bound on the
discretization norm and $\varepsilon_K = \|L - L_K\|$.
This avoids any Newton--Kantorovich argument: the projector error
\emph{alone} controls the eigenvalue error for all $50$ eigenvalues
simultaneously, with no requirement on the eigenvalue gap relative
to $\varepsilon_K$.  For the dominant eigenvalues ($j \leq 15$) the
bound~\eqref{eq:eval-from-proj} is of order $\varepsilon_K \sim 10^{-8}$
(at $K = 48$); for $j \leq 35$ at $K = 256$ it reaches $\sim 10^{-45}$.
The tightest enclosures use $K = 1024$ (2048-bit precision,
$\varepsilon_{1024} \approx 3.23 \times 10^{-180}$): for the leading
eigenvalues ($j \leq 15$) the enclosure radii are of
order $10^{-176}$--$10^{-171}$, degrading to $\sim 10^{-119}$ for
$j = 50$ due to resolvent growth on the excluding circles.
In particular, $|\lambda_2 - \tilde\lambda_2| < 5 \times 10^{-176}$
with $\tilde\lambda_2$ as in Theorem~\ref{thm:main-gkw},
certifying the Wirsing constant to over $175$ decimal digits.
The full list of eigenvalue enclosures is given in the supplementary
material (Table~S3).

The eigenvectors $\hat v_j$
of $L_K$ are the first columns of the ordered Schur unitary $Q_{\mathrm{ord}}$,
known to within the combined Schur and ordschur similarity defect
$\|E\| \leq 5.72 \times 10^{-317}$ at $K = 1024$
(supplementary Table~S6).
At the \emph{infinite-dimensional level}, the projector error
$\vartheta_j$ controls the gap between the true eigenvector $v_j$ of
$L$ and $\hat v_j$ of $L_K$:
$\|v_j - \hat v_j\| \leq 2\vartheta_j / (1 - \vartheta_j)$.

Since $L$ is non-normal, the spectral expansion requires the dual eigenvectors
$\ell_j$.  We compute the spectral coefficient
$\ell_j(\mathbf{1}) = e_1^* P_j e_1$ from the ordered Schur form:
with $T_{\mathrm{ord}} =
  \bigl[\begin{smallmatrix} \lambda_j & T_{12} \\ 0 & T_{22}\end{smallmatrix}\bigr]$
and $q = Q_{\mathrm{ord}}^* e_1$,
\begin{equation}
\label{eq:ell-formula}
\ell_j(\mathbf{1})
  = q_1 - T_{12} \cdot (T_{22} - \lambda_j I)^{-1} q_{\mathrm{rest}},
\end{equation}
where the triangular system is solved in $P$-bit BigFloat arithmetic.
The certified error is dominated by the Schur similarity defect
$\|E\| \leq r_{\mathrm{sch}}\sqrt{1+\delta} + C_A\,\delta$
(Lemma~\ref{lem:schur-defect-to-E}); the triangular solve
residual is negligible ($\sim 10^{-617}$ at 2048-bit).

At $K = 1024$ (2048-bit precision, $\|E\| \leq 5.72 \times 10^{-317}$),
all $50$ spectral coefficients are \emph{sign-certified}:
$|\ell_j(\mathbf{1})| > \delta_j$ with radii
$\delta_j \in [4.5 \times 10^{-316},\; 1.1 \times 10^{-282}]$.
The signs alternate as
$\ell_1 > 0$, $\ell_2 > 0$, $\ell_3 < 0$, $\ell_4 > 0$, \ldots,
with $|\ell_j(\mathbf{1})| \sim |\lambda_j|$ for large~$j$.
Cross-validation at $K = 256$, $512$, and $1024$ confirms agreement
of all certified digits
(see supplementary Table~S6).

\subsection{Certified spectral expansion and computation of $F_n(x)$}
\label{sec:spectral-expansion}

With all $50$ eigenvalues certified simple and the spectral
coefficients $\ell_j(\mathbf{1})$ sign-certified, the spectral
expansion
\begin{equation}
\label{eq:spectral-exp-main}
L^n \mathbf{1}
  = \sum_{j=1}^{N} \lambda_j^n \, \ell_j(\mathbf{1}) \, v_j + R_N(n)
\end{equation}
holds rigorously in $H^2(D_1)$ with the tail bound
\begin{equation}
\label{eq:tail-bound}
\|R_N(n)\| \;\leq\; \rho^{n+1} \cdot M_\infty \cdot \|Q_N \mathbf{1}\|,
\end{equation}
where $Q_N = I - \sum_{j=1}^{N} P_j$ is the tail projector,
$\rho$ is the radius of a separating circle in the eigenvalue gap
$|\lambda_{N+1}| < \rho < |\lambda_N|$, and $M_\infty$ is the certified
resolvent on this circle.  For $N = 50$, the parameters are
$\rho = 1.014 \times 10^{-21}$,
$M_\infty = 6.062 \times 10^{41}$,
and $\|Q_{50}\mathbf{1}\| \leq 6.138 \times 10^{-21}$
(computed at $K = 1024$), giving
$\|R_{50}(n)\| \leq 3.72 \times 10^{21} \cdot (1.01 \times 10^{-21})^{n+1}$.
Already at $n = 1$ the remainder is below $10^{-20}$; at $n = 5$ it drops
below $10^{-104}$ (see supplementary Table~S8).

Define the integrated spectral expansion
\[
F_n^{(j_0)}(x) \;=\; \int_0^x
  \sum_{j=j_0}^{N} \lambda_j^n \,\ell_j(\mathbf{1})\, v_j(t)\, dt,
\]
which is computed exactly in the shifted monomial basis via
$\int_0^x (t-1)^k\,dt = \bigl((x{-}1)^{k+1} - (-1)^{k+1}\bigr)/(k{+}1)$.
For $j_0 = 2$, this measures the integrated deviation of $L^n\mathbf{1}$
from its equilibrium $\ell_1(\mathbf{1})\,v_1$.  The $L^\infty$ error
is controlled by Cauchy--Schwarz and the Gram matrix of the basis on
$[0,1]$ (spectral norm $\leq \pi$):
\[
\sup_{x \in [0,1]} \Bigl|F_n^{(j_0)}(x) - \int_0^x \!\Bigl(L^n\mathbf{1}
  - \textstyle\sum_{j<j_0} \lambda_j^n P_j\mathbf{1}\Bigr) dt \Bigr|
\;\leq\; \sqrt{\pi}\;\rho^{n+1} \cdot C
\;\leq\; 6.6 \times 10^{21} \cdot (1.01 \times 10^{-21})^{n+1}.
\]
Figure~\ref{fig:integrated-spectral} shows $F_n^{(j_0)}$ for
$j_0 = 2, \ldots, 6$ and $n = 1, 2, 5$.  As dominant eigenfunctions are
progressively removed, the Markov partition structure of the Gauss map
becomes increasingly visible, and the amplitude decreases by roughly a
factor $|\lambda_{j_0}/\lambda_{j_0-1}| \approx 1/3$ per removed mode.

\begin{figure}[htbp!]
\centering
\includegraphics[width=0.75\textwidth]{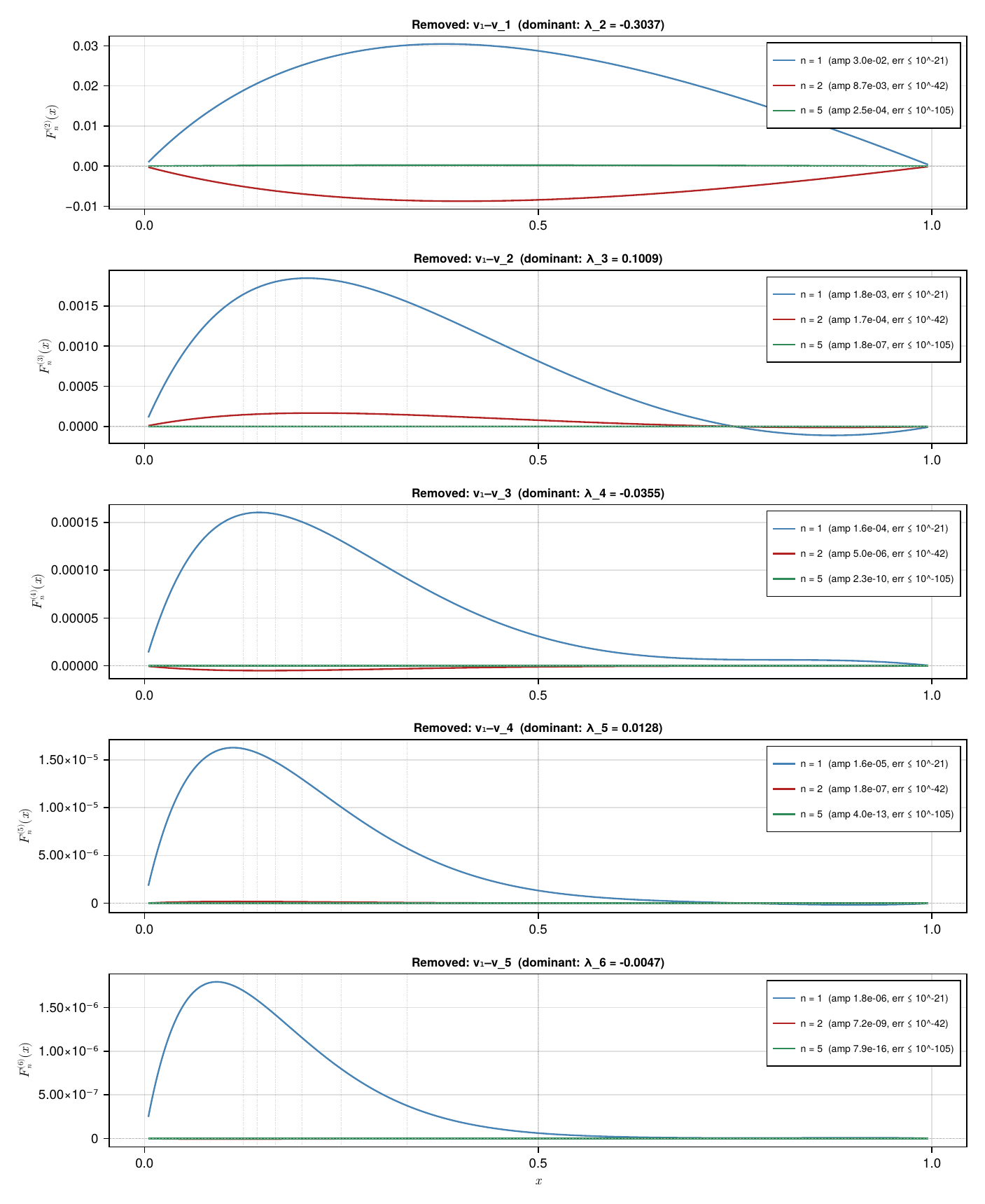}
\caption{Integrated spectral expansion
$F_n^{(j_0)}(x) = \int_0^x \sum_{j \geq j_0}
  \lambda_j^n\,\ell_j(\mathbf{1})\,v_j(t)\,dt$
with the first $j_0 - 1$ eigenfunctions cumulatively removed.
Each panel shows $n = 1$ (blue), $n = 2$ (red), $n = 5$ (green).
Dotted vertical lines mark the Markov partition boundaries $x = 1/k$.
For $j_0 = 6$, $n = 5$, the signal amplitude is ${\sim}\,10^{-16}$
while the $L^\infty$ error bound is ${\leq}\,10^{-105}$, confirming
that the plotted curves are rigorous to plotting accuracy.}
\label{fig:integrated-spectral}
\end{figure}

Returning to the Gauss--Kuzmin problem, the distribution function
$G_n(x) = \int_0^x L^n\mathbf{1}(t)\,dt$ converges to $\log_2(1+x)$.  Since $\lambda_1 = 1$ and
$\ell_1(\mathbf{1})\,v_1 = \frac{1}{\ln 2} \cdot \frac{1}{1+x}$
(the invariant density), the spectral expansion gives
\[
|G_n(x) - \log_2(1{+}x)|
\;\leq\; |F_n^{(2)}(x)| + \sqrt{\pi}\;\|R_N(n)\|.
\]
The dominant error term is $\lambda_2^n\,\ell_2(\mathbf{1})\,\int_0^x v_2$,
which decays at rate $|\lambda_2|^n \approx 0.3037^n$, the
Wirsing constant.  The tail bound certifies that the
$50$-term expansion captures this rate with
rigorous error below $10^{-20}$ already at $n = 1$.
Figure~\ref{fig:spectral-convergence} illustrates the convergence:
the spectral approximant $S_{50}(n,x)/\lambda_1^n$ rapidly approaches
the invariant density $\ell_1(\mathbf{1})\,v_1(x)$, with the spectral
gap $|\lambda_2/\lambda_1| \approx 0.304$ governing the rate.

\begin{figure}[htbp!]
\centering
\includegraphics[width=0.65\textwidth]{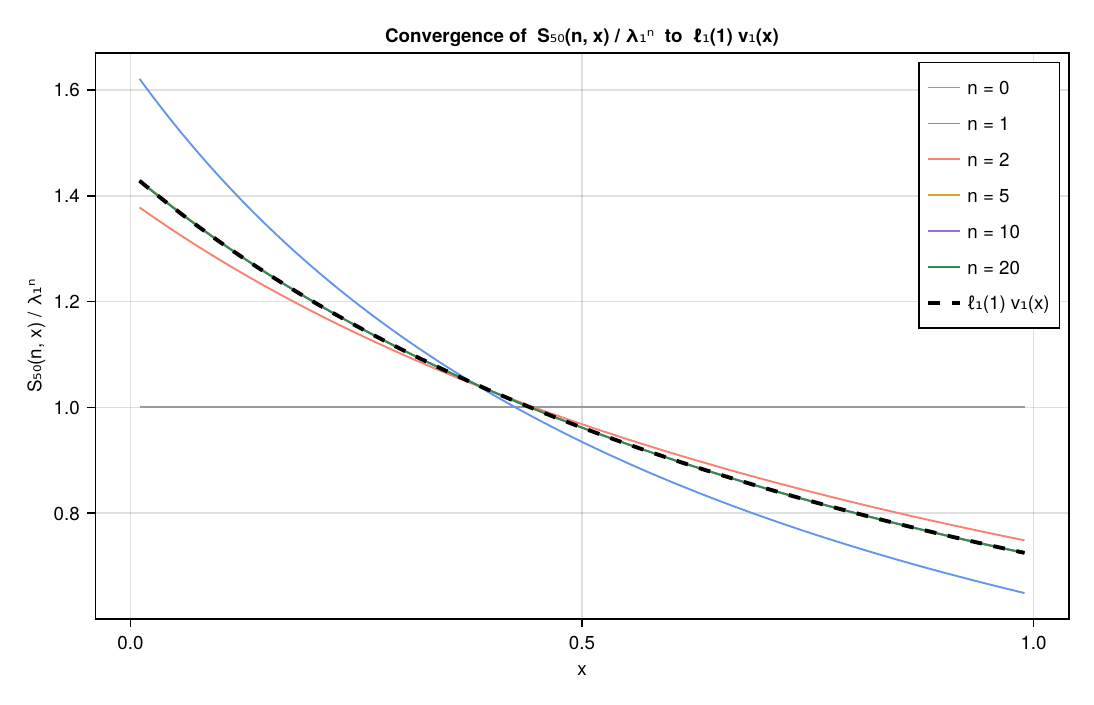}
\caption{Convergence of the spectral approximant
$S_{50}(n,x) = \sum_{j=1}^{50} \lambda_j^n\,\ell_j(\mathbf{1})\,v_j(x)$
(normalized by $\lambda_1^n$) to the invariant density
$\ell_1(\mathbf{1})\,v_1(x) = \frac{1}{\ln 2(1+x)}$ (dashed black).
The spectral gap $|\lambda_2/\lambda_1| \approx 0.304$ governs the
convergence rate: by $n = 5$ the approximation is indistinguishable
from the limit.}
\label{fig:spectral-convergence}
\end{figure}

\section{Discussion and portability}
\label{sec:discussion-portability}

\paragraph{Moving the bottleneck from analysis to certification.}
The central contribution is a shift in \emph{where the hard work lives}.
Once an operator $L$ admits compactness/nuclearity on a suitable space or a
DFLY inequality (Appendix~\ref{app:dfly}),
the remaining spectral analysis reduces to checkable numerical statements about a finite-rank
model $L_N$:
a certified approximation bound $\|L-L_N\|\le \varepsilon_N$,
certified resolvent bounds along chosen contours,
and certified residual information from linear algebra.
Theorems are phrased in a fully \emph{a posteriori} form,
and a large portion of the proof burden is converted into validated numerical linear algebra,
building on~\cite{BlumenthalNisoliTaylorCrush2025}.
A practical consequence is \emph{high-throughput} certification:
many eigenvalues, multiplicities, and Riesz projectors can be validated simultaneously,
precisely in the regime where non-normality and pseudospectral effects
make resolvent-based certification essential~\cite{GohbergKrein}.

\paragraph{The spectral-gap bottleneck and a~posteriori reversal.}
Since a spectral gap is often the hardest ingredient in rigorous
thermodynamic
computations~\cite{PollicottSlipantschuk2024,PollicottVytnova2022,VytnovaWormell2025},
the a~posteriori reversal described in Section~\ref{sec:intro} is of
particular interest.
Liverani~\cite{Liverani2001} proposed certifying spectral data of the true
operator from a finite-rank discretization via
Keller--Liverani perturbation theory~\cite{KellerLiverani1999},
but the approach requires the \emph{strong} resolvent norm of the
discretization, which is difficult to compute when the strong norm is,
e.g., the BV norm.
In the DFLY setting of Appendix~\ref{app:dfly}, we circumvent this by
working with the \emph{weak} resolvent of the discretization, which is computable
via certified singular value enclosures of a finite matrix, and lifting
to strong resolvent bounds via the Lasota--Yorke inequality
(Proposition~\ref{prop:weak-to-strong-bridge}; see
Remark~\ref{rem:why-weak-resolvent} for a detailed comparison).
This makes the a~posteriori approach both complementary to
a~priori methods on holomorphic function
spaces~\cite{BandtlowJenkinson2007,BandtlowJenkinson2008}
or structured approximation
schemes~\cite{BandtlowPohlSchickWeisse2021,BandtlowSlipantschuk2020,SlipantschukEtAl2013},
and fully implementable.

\paragraph{Full spectral convergence in the DFLY setting.}
A second main contribution, developed in Appendix~\ref{app:dfly}, is the
convergence guarantee of Theorem~\ref{thm:spectral-convergence-dfly}:
under a DFLY inequality and a mild rate condition on the
discretization error, every isolated eigenvalue, eigenvector, and Riesz
projector outside the essential spectral radius can be approximated to
arbitrary precision by the finite-rank data of $L_k$.
This upgrades the qualitative spectral stability guaranteed by the
Keller--Liverani theorem~\cite{KellerLiverani1999}
(which itself extends Li's Ulam-conjecture resolution~\cite{Li1976}
beyond the invariant density) to a fully certified and quantitative
statement.
A key ingredient is the coarse-to-fine propagation of
Proposition~\ref{prop:true-to-fine-weak-resolvent-damped}, which controls
the growth of the weak resolvent $\mathcal R_w(z,L_k)$ as
$k\to\infty$, preventing the spectral pollution that would otherwise
invalidate finer certifications.

\paragraph{Portability.}
The method applies beyond transfer operators.
The three-input interface of Section~\ref{sec:compact-cert}
(approximation bound, resolvent control, perturbation bound)
encompasses compact integral operators,
Koopman/Perron--Frobenius operators with noise-induced compactness,
and Markov operators satisfying Doeblin--Fortet/Harris-type
inequalities~\cite{IonescuTulceaMarinescu1950,HennionHerve2001}.
For the latter, certified spectral gaps yield rigorous
mixing-rate estimates as a direct corollary.
Hypocoercivity in PDE~\cite{Villani2009Hypocoercivity}
provides another natural setting where the same certification template applies.
A particularly natural target is the derivative of a renormalization operator
at a fixed point in the analytic category: the domain gain
($\mathcal A(\Omega)\to\mathcal A(\widetilde\Omega)$ with
$\Omega\Subset\widetilde\Omega$) makes $DR_{f_*}$ compact on the
same Hardy/Bergman-type spaces used here, and the truncation error
decays geometrically by the same mechanism as for the Gauss
map~\cite{Lanford1982,deFariaDeMeloPinto2006,Yampolsky2003}.
The present framework could therefore certify hyperbolicity of
period-doubling and critical circle map renormalization
from computed spectral data of finite-rank discretizations,
systematizing the computer-assisted strategy pioneered
by Lanford~\cite{Lanford1982} with the a~posteriori spectral
tools developed here.

\paragraph{Limitations.}
The method's quantitative strength is governed by the product
$\varepsilon_N \cdot \sup_{z\in\Gamma}\|(zI-L_N)^{-1}\|$:
if this exceeds $1$, the Neumann-series bound fails;
if it is close to $1$, constants become weak, reflecting genuine pseudospectral sensitivity.
The pipeline certifies \emph{isolated} spectral data (clusters separated by contours),
not continuous spectrum.

\paragraph{Reproducibility.}
All certifications were carried out on a single desktop workstation
with no supercomputing resources
(see supplementary material for specifications);
increasing the number of eigenvalues or the precision
scales directly with matrix size and arithmetic precision.
Detailed computational methodology, all certified spectral tables,
eigenvector coefficients, and additional figures are collected
in the supplementary material (arXiv ancillary file).
The full source code for the certification pipeline is available at
\url{https://github.com/orkolorko/GKWExperiments.jl};
certified spectral data in machine-readable form are deposited at
\url{https://doi.org/10.7910/DVN/HKM3Y2}.

\appendix
\section{Strong--weak (DFLY) resolvent lifting}\label{app:dfly}

This appendix extends the abstract compact template of
Section~\ref{sec:compact-cert} to the common \emph{strong-weak} transfer-operator setting:
a Banach scale
$(\mathcal B_s,\|\cdot\|_s)\hookrightarrow(\mathcal B_w,\|\cdot\|_w)$
with compact embedding, and a Doeblin--Fortet--Lasota--Yorke (DFLY) inequality.
The same resolvent-based certification logic applies:
validated \emph{weak} resolvent bounds for finite-rank discretizations yield
certified \emph{strong} resolvent bounds for the true operator,
certified spectral exclusion curves, and
a coarse-fine workflow in which a coarse certificate propagates to arbitrarily fine levels.

We use the resolvent notation of Section~\ref{subsec:resolvent-householder},
writing $\mathcal R_s(z,A):=\|R_A(z)\|_{s\to s}$ and
$\mathcal R_w(z,A):=\|R_A(z)\|_{w\to w}$ to distinguish operator norms.

\subsection{Setting: Banach scale, DFLY bounds, and Galerkin hierarchy}\label{subsec:dfly-setting}

\begin{assumption}\label{ass:dfly-scale}
Let $(\mathcal B_s,\|\cdot\|_s)\hookrightarrow(\mathcal B_w,\|\cdot\|_w)$ be a continuous embedding
with embedding constant $E_{s\to w}:=\|i\|_{s\to w}$.
Assume the embedding is compact.
\end{assumption}

\begin{assumption}\label{ass:dfly-setting}
Let $L\in\mathcal L(\mathcal B_s)$ admit a bounded extension to $\mathcal B_w$
satisfying the one-step DFLY bounds:
for constants $a,b,M\in(0,\infty)$ and all $u\in\mathcal B_s$,
\begin{equation}\label{eq:dfly-L}
\|Lu\|_s\le a\|u\|_s+b\|u\|_w,
\qquad
\|Lu\|_w\le M\|u\|_w.
\end{equation}
Assume we are given finite-rank operators $L_k\in\mathcal L(\mathcal B_s)$
extending boundedly to $\mathcal B_w$ with the same one-step bounds:
\begin{equation}\label{eq:dfly-Lk}
\|L_k u\|_s\le a\|u\|_s+b\|u\|_w,
\qquad
\|L_k u\|_w\le M\|u\|_w.
\end{equation}
Define the mixed discrepancy $\delta_k:=\|L-L_k\|_{s\to w}$.
\end{assumption}

\begin{assumption}\label{ass:dfly-hierarchy}
Assume $L_k$ arise from a nested Galerkin hierarchy:
finite-rank projections $\pi_k:\mathcal B_w\to\mathcal B_w$ with
$\pi_k^2=\pi_k$, $\|\pi_k\|_{w\to w}\le 1$, $\|\pi_k\|_{s\to s}\le 1$,
$\pi_k\pi_{k'}=\pi_k$ for $k'\ge k$,
$\mathrm{Ran}(\pi_k)\subset\mathcal B_s$, and $L_k:=\pi_k\,L\,\pi_k$.
Assume there exists a computable sequence $\varepsilon_k\downarrow0$ such that
\begin{equation}\label{eq:mixed-tail}
\|(\mathrm{Id}-\pi_k)L\|_{s\to w}+\|L(\mathrm{Id}-\pi_k)\|_{s\to w}\le \varepsilon_k.
\end{equation}
In particular,
$\|L-L_k\|_{s\to w}\le \varepsilon_k$ and
$\|L_{k'}-L_k\|_{s\to w}\le \varepsilon_k$ for $k'\ge k$.
\end{assumption}

\begin{lemma}\label{lem:dfly-iterate}
Under Assumption~\ref{ass:dfly-setting}, for every $n\ge1$ and $u\in\mathcal B_s$,
\begin{equation}\label{eq:dfly-Ln}
\|L^{n}u\|_s\le a^{n}\|u\|_s+b_n\|u\|_w,
\qquad
\|L^{n}u\|_w\le M^{n}\|u\|_w,
\end{equation}
where
$b_n:=b\sum_{j=0}^{n-1}a^{\,n-1-j}M^{\,j}$.
The same bounds hold with $L$ replaced by $L_k$.
\end{lemma}

\begin{proof}
Iterate the weak bound and insert $\|L^j u\|_w\le M^j\|u\|_w$ into the strong inequality.
\end{proof}

\begin{lemma}\label{lem:dfly-resolvent-ineq}
Under Assumption~\ref{ass:dfly-setting}, if $(zI-L)u=f$ with $u,f\in\mathcal B_s$, then
\begin{equation}\label{eq:dfly-resolvent-ineq}
(|z|-a)\,\|u\|_s\ \le\ b\,\|u\|_w+\|f\|_s.
\end{equation}
The same holds with $L$ replaced by $L_k$.
\end{lemma}

\begin{proof}
From $zu=Lu+f$ and \eqref{eq:dfly-L}: $|z|\|u\|_s\le a\|u\|_s+b\|u\|_w+\|f\|_s$.
\end{proof}

\begin{remark}\label{rem:fredholm-dfly}
Under Assumptions~\ref{ass:dfly-scale}--\ref{ass:dfly-setting}, the essential spectral radius of
$L:\mathcal B_s\to\mathcal B_s$ is $\le a$ (Hennion-type quasi-compactness).
We use only: for $|z|>a$, injectivity of $zI-L$ implies invertibility on $\mathcal B_s$.
\end{remark}

\subsection{Weak-to-strong resolvent lifting and spectral exclusion}\label{subsec:dfly-bridge}

\begin{proposition}\label{prop:weak-to-strong-bridge}
Assume Assumption~\ref{ass:dfly-scale}--\ref{ass:dfly-setting}. Fix $k\ge1$ and $z\in\mathbb C$ with
$z\notin\sigma(L_k)$, and set $M_k(z):=\mathcal R_w(z,L_k)$.
If
\begin{equation}\label{eq:bridge-condition}
|z|\ >\ a+b\,M_k(z)\,\delta_k,
\end{equation}
then $z\notin\sigma(L)$ and
\begin{equation}\label{eq:strong-resolvent-bound}
\mathcal R_s(z,L)
\ \le\
\frac{b\,M_k(z)\,E_{s\to w}+1}{|z|-a-b\,M_k(z)\,\delta_k}.
\end{equation}
\end{proposition}

\begin{proof}
\emph{Injectivity.}\
Let $(zI-L)u=0$. Then $(zI-L_k)u=(L-L_k)u$, hence
$u=R_{L_k}(z)(L-L_k)u$.
Taking weak norms gives
$\|u\|_w\le M_k(z)\delta_k\|u\|_s$.
Lemma~\ref{lem:dfly-resolvent-ineq} with $f=0$ yields
$(|z|-a)\|u\|_s\le b\,M_k(z)\delta_k\|u\|_s$,
forcing $u=0$ under \eqref{eq:bridge-condition}.
Since $|z|>a$, Remark~\ref{rem:fredholm-dfly} gives bijectivity.

\emph{Quantitative bound.}\
For $f\in\mathcal B_s$ with $\|f\|_s=1$, set $u:=R_L(z)f$.
From $(zI-L_k)u=f+(L-L_k)u$, taking weak norms yields
$\|u\|_w\le M_k(z)E_{s\to w}+M_k(z)\delta_k\|u\|_s$.
Lemma~\ref{lem:dfly-resolvent-ineq} applied to $(zI-L)u=f$ gives
$(|z|-a)\|u\|_s \le b\|u\|_w+1$;
inserting the weak bound and rearranging yields \eqref{eq:strong-resolvent-bound}.
\end{proof}

\begin{remark}\label{rem:why-weak-resolvent}
The use of the \emph{weak} resolvent $\mathcal R_w(z,L_k)$ of the
\emph{discretization} in Proposition~\ref{prop:weak-to-strong-bridge}
is essential for two reasons.

First, it is computable: since $L_k$ is finite rank, $\mathcal R_w(z,L_k)$
reduces to a certified singular value problem for a finite matrix.
An approach based on the strong resolvent $\mathcal R_s(z,L_k)$
(as proposed in the feasibility study~\cite{Liverani2001})
requires the induced $\|\cdot\|_s$-norm of a matrix, which is
nontrivial when $\|\cdot\|_s$ is, e.g., the BV norm.

Second, $\mathcal R_w(z,L)$ for the \emph{original} operator is not available:
in typical DFLY applications $L$ is a weak contraction
($\|L\|_{w\to w}\le 1$), so its weak essential spectral radius is~$1$
and $\mathcal R_w(z,L)=+\infty$ for every $|z|\le 1$, precisely
where the discrete eigenvalues of interest lie.
A strong-weak Neumann-series argument using $\mathcal R_w(z,L)$ therefore
fails at exactly the points it would be needed.
The DFLY inequality resolves this: it converts
the computable quantity $\mathcal R_w(z,L_k)$ (finite because $L_k$
has no essential spectrum) into a certified bound on
$\mathcal R_s(z,L)$ (finite because the eigenvalues are discrete in
$\mathcal B_s$), bypassing $\mathcal R_w(z,L)$ entirely.
\end{remark}

\begin{proposition}\label{prop:exclude-by-bridge-dfly}
Assume Assumption~\ref{ass:dfly-scale}--\ref{ass:dfly-setting}.
Let $\Gamma$ be a piecewise $C^1$ closed curve with $\Gamma\subset\rho(L_k)$ and
\begin{equation}\label{eq:bridge-on-curve}
\inf_{z\in\Gamma}\bigl(|z|-a\bigr)\;>\; b\,\delta_k\,\sup_{z\in\Gamma} \mathcal R_w(z,L_k).
\end{equation}
Then $\Gamma\subset\rho(L)$ and
$\sup_{z\in\Gamma}\mathcal R_s(z,L)\le
\sup_{z\in\Gamma}\frac{b\,\mathcal R_w(z,L_k)\,E_{s\to w}+1}{|z|-a-b\,\mathcal R_w(z,L_k)\,\delta_k}$.
In particular, if $\Gamma=\partial U$ then $P_L(U)$ is well-defined on $\mathcal B_s$.
\end{proposition}

\begin{proof}
Apply Proposition~\ref{prop:weak-to-strong-bridge} pointwise on $\Gamma$.
\end{proof}

\begin{lemma}\label{lem:strong-from-weak-Lk}
If $L_k$ satisfies \eqref{eq:dfly-Lk} and $z\in\rho(L_k)$ with $|z|>a$, then
\begin{equation}\label{eq:strong-from-weak-Lk}
\mathcal R_s(z,L_k)
\ \le\
\frac{b\,\mathcal R_w(z,L_k)\,E_{s\to w}+1}{|z|-a}.
\end{equation}
\end{lemma}

\begin{proof}
Apply Lemma~\ref{lem:dfly-resolvent-ineq} to $u:=R_{L_k}(z)f$ and bound $\|u\|_w$ via $\mathcal R_w(z,L_k)$.
\end{proof}

\begin{lemma}\label{lem:Rk-minus-R}
If $z\in\rho(L)\cap\rho(L_k)$, then
\begin{equation}\label{eq:Rk-minus-R}
\|R_{L_k}(z)-R_L(z)\|_{s\to w}
\ \le\
\mathcal R_w(z,L_k)\,\delta_k\,\mathcal R_s(z,L).
\end{equation}
\end{lemma}

\begin{proof}
From $R_{L_k}(z)-R_L(z)=R_{L_k}(z)(L-L_k)R_L(z)$, bound $\|(L-L_k)v\|_w\le\delta_k\|v\|_s$.
\end{proof}

\subsection{Coarse--fine propagation and weak resolvent growth}\label{subsec:dfly-coarse-fine}

A key issue in the DFLY framework is to prevent the weak resolvent
$\mathcal R_w(z,L_k)$ from exploding as $k$ increases.
The next results extend the resolvent-blowup control
of~\cite{BlumenthalNisoliTaylorCrush2025} to the present abstract setting,
providing \emph{true-to-fine} control: once $z\in\rho(L)$ and
$\mathcal R_s(z,L)$ are certified at a coarse level,
$\mathcal R_w(z,L_k)$ is explicitly controlled for all finer $k$.

\begin{assumption}\label{ass:Ek}
For each $k$, set $E_k:=\mathrm{Ran}(\pi_k)$ and assume the norm equivalence constant
$E_{k,w\to s}:=\sup_{u\in E_k\setminus\{0\}}\|u\|_s/\|u\|_w$
is explicitly bounded.
\end{assumption}

\begin{lemma}\label{lem:dfly-damps-E}
Under Assumptions~\ref{ass:dfly-setting}--\ref{ass:Ek},
for every $k$, $N\ge1$, and $f\in E_k$,
\begin{equation}\label{eq:LkN-w2s}
\|L_k^{N}f\|_s
\ \le\
\bigl(a^{N}E_{k,w\to s}+b_N\bigr)\|f\|_w.
\end{equation}
\end{lemma}

\begin{proof}
Apply Lemma~\ref{lem:dfly-iterate} to $f\in E_k\subset\mathcal B_s$ and use $\|f\|_s\le E_{k,w\to s}\|f\|_w$.
\end{proof}

\begin{proposition}\label{prop:true-to-fine-weak-resolvent-damped}
Assume Assumption~\ref{ass:dfly-scale}--\ref{ass:dfly-hierarchy} and Assumption~\ref{ass:Ek}.
Fix $k\ge1$ and $z\in\mathbb C$ with $|z|>a$ and $z\in\rho(L)$.
Write $K(z):=\mathcal R_s(z,L)$ and $\Delta_k:=\|L-L_k\|_{s\to w}$.

Let $N\ge1$ and define
\begin{equation}\label{eq:SN-true-def}
S_N^{(k)}(z):=\frac{1}{|z|}\sum_{\ell=0}^{N-1}\frac{\|L_k^\ell\|_{w\to w}}{|z|^\ell}
\ \le\
\frac{1}{|z|}\sum_{\ell=0}^{N-1}\Bigl(\frac{M}{|z|}\Bigr)^\ell ,
\end{equation}
and
\begin{equation}\label{eq:betaN-true}
\widetilde\beta_N(z)
:=
\frac{1}{|z|^{N}}\,
\Delta_k\,K(z)\,\bigl(a^{N}E_{k,w\to s}+b_N\bigr).
\end{equation}
If $\widetilde\beta_N(z)<1$, then $z\in\rho(L_k)$ and
\begin{equation}\label{eq:true-to-fine-Mk}
\mathcal R_w(z,L_k)
\ \le\
\frac{
S_N^{(k)}(z)
+\;|z|^{-N}\,E_{s\to w}K(z)\,\bigl(a^{N}E_{k,w\to s}+b_N\bigr)
}{
1-\widetilde\beta_N(z).
}
\end{equation}
\end{proposition}

\begin{proof}
Fix $f\in\mathcal B_w$ and consider $u\in\mathcal B_s$ solving $(zI-L_k)u=f$.
(Existence will follow from the a priori bound below.)
Rewrite the equation as
\[
(zI-L)u=f+(L-L_k)u.
\]
Applying $R_L(z)$ yields
\[
u=R_L(z)f+R_L(z)(L-L_k)u.
\]
Taking weak norms and using $\|(L-L_k)u\|_w\le \Delta_k\|u\|_s$ gives
\begin{equation}\label{eq:u-w-bound}
\|u\|_w\le E_{s\to w}K(z)\|f\|_s+E_{s\to w}K(z)\Delta_k\|u\|_s.
\end{equation}

Next, apply the truncated Laurent identity: from $(zI-L_k)u=f$,
\[
u
=
\frac{1}{z}\sum_{\ell=0}^{N-1}\frac{L_k^{\ell}f}{z^{\ell}}
+\frac{1}{z^{N}}\,L_k^{N}u.
\]
Taking weak norms yields
\[
\|u\|_w\le S_N^{(k)}(z)\,\|f\|_w+|z|^{-N}\,\|L_k^N u\|_w.
\]
Using $\|L_k^N u\|_w\le E_{s\to w}\|L_k^N u\|_s$ and Lemma~\ref{lem:dfly-damps-E}
(with $L_k^N$ landing in the finite-dimensional range), we get
\[
\|L_k^N u\|_w
\le E_{s\to w}\bigl(a^N\|u\|_s+b_N\|u\|_w\bigr).
\]
Insert this into the previous inequality and rearrange to isolate $\|u\|_w$:
\[
\Bigl(1-|z|^{-N}E_{s\to w}b_N\Bigr)\|u\|_w
\le
S_N^{(k)}(z)\|f\|_w
+|z|^{-N}E_{s\to w}a^N\|u\|_s.
\]
Now substitute the bound \eqref{eq:u-w-bound} into the right-hand side through $\|u\|_s$
and collect terms so that the coefficient of $\|u\|_w$ becomes
$1-\widetilde\beta_N(z)$; one obtains
\[
\|u\|_w
\le
\frac{
S_N^{(k)}(z)\|f\|_w
+|z|^{-N}E_{s\to w}K(z)\bigl(a^{N}E_{k,w\to s}+b_N\bigr)\|f\|_w
}{
1-\widetilde\beta_N(z)
},
\]
where we also used $\|f\|_s\le E_{k,w\to s}\|f\|_w$ when $f\in E_k$.
This proves that $(zI-L_k)$ is injective with bounded inverse on the finite-dimensional range,
hence $z\in\rho(L_k)$ and \eqref{eq:true-to-fine-Mk} holds.
\end{proof}

\begin{remark}\label{rem:true-to-fine-workflow}
In practice one certifies $z\in\rho(L)$ and $K(z)$ on a contour $\Gamma$
using Proposition~\ref{prop:weak-to-strong-bridge} at a \emph{coarse} level $k_0$.
Proposition~\ref{prop:true-to-fine-weak-resolvent-damped} then bounds
$\mathcal R_w(z,L_k)$ for all finer $k\ge k_0$, without recomputing resolvents.
\end{remark}

\begin{corollary}\label{cor:weak-resolvent-growth-generalM}
Assume Assumption~\ref{ass:dfly-scale}--\ref{ass:dfly-hierarchy} and Assumption~\ref{ass:Ek},
with $0<a<M$.
Fix $\mu\in(a,M)$ and set
\[
\mathfrak q:=\frac{\bigl|\log(\mu/M)\bigr|}{\bigl|\log(a/M)\bigr|}\in(0,1),
\qquad
C_*:=1+\frac{b}{M-a}.
\]
Let $k\ge1$ and assume $z\in\rho(L)$ satisfies $|z|\ge \mu$.
Write $K(z):=\mathcal R_s(z,L)$ and $\delta_k:=\|L-L_k\|_{s\to w}$.
Choose
$N_k:=\lceil \log E_{k,w\to s}/|\log(a/M)|\rceil$.
If
\begin{equation}\label{eq:denom-growth-pos-generalM}
1-\frac{M}{\mu}\,\delta_k\,K(z)\,C_*\,E_{k,w\to s}^{\mathfrak q}>0,
\end{equation}
then $z\in\rho(L_k)$ and
\begin{equation}\label{eq:weak-resolvent-growth-generalM}
\mathcal R_w(z,L_k)
\ \le\
\frac{\dfrac{M}{\mu}\,E_{k,w\to s}^{\mathfrak q}}{1-\dfrac{M}{\mu}\,\delta_k\,K(z)\,C_*\,E_{k,w\to s}^{\mathfrak q}}
\left(
\frac{1}{M-\mu}+E_{s\to w}\,K(z)\,C_*
\right).
\end{equation}
\end{corollary}

\begin{proof}
Apply Proposition~\ref{prop:true-to-fine-weak-resolvent-damped} with $N=N_k$.
The choice ensures $a^{N_k}E_{k,w\to s}\le M^{N_k}$ and
$|z|^{-N_k}\le (M/\mu)E_{k,w\to s}^{\mathfrak q}$ by logarithmic interpolation.
Substituting $b_{N_k}\le \frac{b}{M-a}M^{N_k}$
into \eqref{eq:true-to-fine-Mk} yields \eqref{eq:weak-resolvent-growth-generalM}.
\end{proof}

\subsection{Spectral convergence in the DFLY setting}\label{subsec:dfly-convergence}

\begin{corollary}\label{cor:resolvent-strong-to-weak-convergence}
In the setting of Corollary~\textup{\ref{cor:weak-resolvent-growth-generalM}},
suppose that at some coarse level $k_0$ the perturbation condition
\eqref{eq:bridge-condition} certifies $z\in\rho(L)$ and $K(z)<\infty$.
If
\begin{equation}\label{eq:feasibility-deltaEkq}
\delta_k\,E_{k,w\to s}^{\mathfrak q}\ \longrightarrow\ 0
\qquad\text{as }k\to\infty,
\end{equation}
then $\|R_{L_k}(z)-R_L(z)\|_{s\to w}\to 0$:
Corollary~\textup{\ref{cor:weak-resolvent-growth-generalM}} gives
$\mathcal R_w(z,L_k)=O(E_{k,w\to s}^{\mathfrak q})$,
and Lemma~\textup{\ref{lem:Rk-minus-R}} yields the claim.
\end{corollary}

The following theorem is the DFLY analogue of
Theorem~\ref{thm:spectral-convergence-compact}:
it shows that under the feasibility condition \eqref{eq:feasibility-deltaEkq},
the full discrete spectral picture of $L$ outside the essential spectral radius
can be recovered from the Galerkin hierarchy $L_k$.

\begin{theorem}\label{thm:spectral-convergence-dfly}
Under Assumption~\ref{ass:dfly-scale}--\ref{ass:dfly-hierarchy} and Assumption~\ref{ass:Ek},
suppose $0<a<M$.
Fix $\mu\in(a,M)$ and let $\mathfrak q:=|\log(\mu/M)|/|\log(a/M)|\in(0,1)$.
Suppose
\begin{equation}\label{eq:spectral-conv-rate}
\delta_k\,E_{k,w\to s}^{\mathfrak q}\;\longrightarrow\; 0
\qquad\text{as }k\to\infty.
\end{equation}
Let $U\subset\mathbb C$ be a bounded domain
with piecewise $C^1$ boundary $\Gamma:=\partial U$ such that
$\inf_{z\in\Gamma}|z|\ge\mu$,
$\Gamma\subset\rho(L)$, and $\sigma(L)\cap\Gamma=\emptyset$.
Then there exists $k_0$ such that for all $k\ge k_0$:
\begin{enumerate}
\item $\Gamma\subset\rho(L_k)$ and the perturbation condition
$\sup_{z\in\Gamma}b\,\delta_k\,\mathcal R_w(z,L_k)<
\inf_{z\in\Gamma}(|z|-a)$ holds;
\item $\dim\operatorname{Ran}P_{L_k}(U)=\dim\operatorname{Ran}P_L(U)$;
\item $\|P_{L_k}(U)-P_L(U)\|_{s\to w}\to 0$ as $k\to\infty$.
\end{enumerate}
In particular, every isolated eigenvalue, eigenvector, and Riesz projector
of $L$ outside the essential spectral radius can be approximated to arbitrary
precision by the finite-rank data of $L_k$, provided the contour can be
chosen with $\inf|z|\ge\mu$ for some $\mu$ satisfying \eqref{eq:spectral-conv-rate}.
\end{theorem}

\begin{proof}
Since $\inf_{z\in\Gamma}|z|\ge\mu>a$ and $\Gamma\subset\rho(L)$,
Corollary~\ref{cor:weak-resolvent-growth-generalM} applies to every $z\in\Gamma$.
By Corollary~\ref{cor:resolvent-strong-to-weak-convergence},
the rate condition \eqref{eq:spectral-conv-rate} ensures that
$\Gamma\subset\rho(L_k)$ for $k\ge k_0$ and
$\|R_{L_k}(z)-R_L(z)\|_{s\to w}\to 0$ uniformly on $\Gamma$.
In particular, the perturbation condition of
Proposition~\ref{prop:exclude-by-bridge-dfly} is eventually satisfied on all of $\Gamma$,
giving (i).
Claims (ii) and (iii) follow from the contour integral representation
$P_{L_k}(U)-P_L(U)=\frac{1}{2\pi i}\oint_\Gamma(R_{L_k}(z)-R_L(z))\,dz$
and the uniform convergence of the integrand in $\|\cdot\|_{s\to w}$
established in Corollary~\ref{cor:resolvent-strong-to-weak-convergence}.
\end{proof}

\section*{Acknowledgements}
The author would like to acknowledge helpful discussions with
A.~Blumenthal, O.~Bandtlow, S.~Galatolo, and W.~Tucker.
The idea of treating the Gauss map arose during a discussion with I.~Melbourne.
The author is also grateful to the participants of related seminars and workshops
for valuable feedback on preliminary versions of this work.

Claude (Anthropic) and ChatGPT (OpenAI) were used as assistants for prose editing
and code development.

The author thanks the Uppsala University Library for access to their edition of
Gauss's \emph{Werke}, the Maurizio Monge Foundation for supporting access to
Claude Pro, and Prof.~Y.~Sato for access to computational resources.

I.N.\ was partially supported by the Postgraduate Program in Mathematics at UFRJ,
CAPES--Finance Code 001, CNPq Projeto Universal No.\ 404943/2023-3,
CAPES--PRINT No.\ 88881.311616/2018-00, and CAPES--STINT No.\ 88887.155746/2017-00.

Any remaining errors are solely the author's responsibility.

\bibliographystyle{plainurl}
\bibliography{refs}
\end{document}